\pdfoutput=1

\documentclass{scrartcl}
\usepackage[top=20truemm,bottom=30truemm,left=30truemm,right=30truemm]{geometry}
\usepackage{amsmath,amssymb}
\usepackage{mathtools}
\usepackage[dvipdfmx]{graphicx}
\usepackage{amsthm}
\usepackage{bm}

\theoremstyle{plain}
\newtheorem{thm}{Theorem}[section]
\newtheorem{lemma}[thm]{Lemma}
\newtheorem{prop}[thm]{Proposition}
\newtheorem{cor}[thm]{Corollary}
\theoremstyle{remark}
\newtheorem{rem}[thm]{Remark}
\theoremstyle{definition}

\newcommand{\dist}{\operatorname{dist}}

\newcommand{\dps}{\displaystyle}

\usepackage{xcolor}
\definecolor{cyan20}{cmyk}{.2,0,0,0}

\newcommand{\tn}[1]{{\color{black}#1}}

\newcommand{\dual}[1]{\left\langle#1\right\rangle}
\newcommand{\vnorm}[1]{|\hspace{-0.3mm}|\hspace{-0.3mm}|#1|\hspace{-0.3mm}|\hspace{-0.3mm}|}

\begin{document}
%
%

\title{A mass-lumping finite element method for radially symmetric solution of a multidimensional semilinear heat equation with blow-up}

\author{
Toru Nakanishi\thanks{Graduate School of Mathematical Sciences, The University of Tokyo, Komaba 3-8-1, Meguro-ku, Tokyo 153-8914, Japan. \textit{E-mail}: \texttt{nakanish@ms.u-tokyo.ac.jp}} 
\and 
Norikazu Saito\thanks{Graduate School of Mathematical Sciences, The University of Tokyo, Komaba 3-8-1, Meguro-ku, Tokyo 153-8914, Japan. \textit{E-mail}: \texttt{norikazu@g.ecc.u-tokyo.ac.jp}}
}


\maketitle


\begin{abstract}
This study presents a new mass-lumping finite element method for computing the radially symmetric solution of a semilinear heat equation in an $N$ dimensional ball \tn{($N\ge 2$)}. We provide two schemes, (ML--1) and (ML--2), and derive \tn{their} error estimates \tn{through} the discrete maximum principle. \tn{In the weighted $L^{2}$ norm,} the convergence of (ML--1) \tn{was} at the optimal order \tn{but} that of (ML--2) \tn{was} only at sub-optimal order. Nevertheless, \tn{scheme} (ML--2) \tn{reproduces a} blow-up of the solution of the original equation. In fact, \tn{in scheme} (ML--2), we \tn{could} accurately \tn{approximate} the blow-up time. \tn{Our theoretical results were validated in} numerical \tn{experiments}.  
\end{abstract}


\bigskip


\section{Introduction}
\label{sec:intro}

This paper \tn{applies} the finite element method (FEM) to a semilinear parabolic equation with a singular convection term:  
\begin{subequations}
\label{eq:1}
\begin{align}
&u_{t}=u_{xx}+\frac{N-1}{x}u_{x}+f(u), 
 &&x\in I=(0,1),~ t>0, \label{eq:1a}\\
&u_x (0,t)=u(1,t)=0, &&t>0, \label{eq:1b}\\
&u(x,0)=u^0(x),    &&x\in I. \label{eq:1c}
\end{align}
\end{subequations}
\tn{Here}, 
$u=u(x,t)$, $x\in \overline{I}=[0,1], t\ge 0$ denotes the function to be \tn{found}, $f$ \tn{is} a given locally Lipschitz continuous function\tn{,} and $u^0$ \tn{is} a given continuous function. Throughout this paper, we assume that 
\begin{equation}
\label{eq:N}
\quad \mbox{$N$ is an integer}
\ge 2. 
\end{equation}
\tn{To compute the} blow-up solution of \eqref{eq:1}, \tn{we apply} 
Nakagawa's time-increment control strategy (see \cite{nak76} and Section \ref{sec:7} of the present paper), a powerful technique for \tn{approximating} blow-up \tn{times}.
As recalled below, the standard finite element approximation is \tn{unuseful} for achieving this purpose. 
\tn{We thus} propose a new mass-lumping finite element approximation, prove its convergence, and \tn{apply it to} a blow-up analysis. 

\tn{We first} clarify the motivation of this study. 
In many engineering problems, the \tn{spatial} dimension of a mathematical model is at most three. \tn{Solving} partial 
differential equations (PDEs) \tn{in more than three spatial
dimensions} is usually motivated by 
mathematical interests. \tn{Mathematicians understand that solving} problems in a general setting \tn{can} reveal the hidden natures of PDEs. \tn{One} successful 
result \tn{is} the discovery of Fujita's blow-up exponent for \tn{the} 
semilinear heat equation \tn{of} $U=U(\bm{x},t)$ given as 
\begin{equation}
\label{eq:F}
U_{t}=\Delta U+ f(U)\qquad (\bm{x}\in\mathbb{R}^N,~t>0) ,
\end{equation}
where $N$ and $f(U)$ are \tn{defined} above.  
Assuming $f(U)=U|U|^{\alpha}$ with $\alpha>0$, Fujita showed that any positive solution
blows up in finite time if $1+\alpha< 1+2/N$, \tn{but} a solution remains 
smooth at any time \tn{if the} small initial value \tn{is small and} $1+\alpha>1+2/N$.
\tn{The} quantity $p_c=1+2/N$ is known as Fujita's critical exponent\tn{,} 
and \tn{Eq.}\eqref{eq:F} is called Fujita's equation. 
\tn{Since} Fujita's work, a huge number of studies have been 
devoted to critical phenomena in nonlinear PDEs of several kinds \tn{(see \cite{dl00,lev90,qs19} for details.)} 
\tn{The} knowledge \tn{gained by these studies has been applied to} problems 
\tn{with spatial} dimensions \tn{of} three \tn{or fewer}. 
\tn{However}, many problems related to stochastic 
analysis are formulated \tn{as} higher-dimensional PDEs. 
\tn{These problems have attracted much interest,} but are beyond \tn{the scope }of \tn{the present} study.  

\tn{Non-stationary} problems \tn{in four dimensional space are }difficult \tn{to solve by numerical methods}, 
even \tn{on} modern computers. Consequently, numerical \tn{analyses} 
of \tn{the} blow-up solutions of nonlinear PDEs \tn{have been} restricted to two-dimensional space 
(see \tn{for example}
\cite{alm01,bgr04,bqr05,cgkm16,cho13,cho16,cho18,fgr03,gro06,hmr08,gr01,iho02,it00,nu77,nb08} ). 
Although Nakagawa's time-increment control strategy is applied to various nonlinear PDEs including the nonlinear heat, wave and Schr{\" o}dinger equations, these equations are considered only in the one-space dimension; see \cite{che86,cho10,cho13,cho07,iho02,ss16,ss16b}.
We know two notably exceptions; the one is \cite{nu77} where 
the finite element method to a semilinear heat equation 
in a two dimensional polygonal domain was considered, and the other is \cite{che92} where the finite difference method to the radially symmetric solution of \tn{the} semilinear heat equation in \tn{an} $N$ dimensional ball was studied. 

Following \cite{che92}, \tn{the present paper investigates} radially symmetric solutions \tn{to Eq.\eqref{eq:F}}. \tn{Assuming radial symmetry of the solution and the }given data, the $N$-dimensional equation \tn{reduces to a} one-dimensional equation. 
More specifically, considering \eqref{eq:F} in an $N$-dimensional unit ball $\textup{B}=\{\bm{x}\in\mathbb{R}^N\mid |\bm{x}|_{\mathbb{R}^N}<1\}$ with the homogeneous Dirichlet boundary condition on the boundary and assuming $U$ is expressed as $u(x)=U(\bm{x})$ for $\bm{x}\in \textup{B}$ and $x=|\bm{x}|_{\mathbb{R}^N}$, we came to consider the problem \eqref{eq:1}.  

After \tn{completing} the present work, we learned \tn{that} Cho and Okamoto \cite{co20} \tn{extended the work in} \cite{che92}.   
\tn{The time dimension was discretized by the} semi-implicit Euler method \tn{in \cite{che92}, but Cho and Okamoto \cite{co20} explored the} explicit scheme\tn{, then proved optimal-order convergence \tn{with} Nakagawa's strategy.
Because \tn{their} schemes use special approximations around the origin to maintain some analytical properties of the solution, \tn{they should be performed on a uniform} spatial mesh. \tn{Conversely}, when \tn{seeking} the blow up solution, non-uniform partitions of the space variable \tn{are} useful for examining highly concentrated solutions at the origin. \tn{For this purpose}, we \tn{developed} the FEM scheme. 

FEM analyses of the linear case\tn{, in which} $f(u)=0$ in Eq.\eqref{eq:1}} is replaced by a given function $f(x,t)$, \tn{are not new}. 
Eriksson and Thom\'{e}e \cite{et84} and Thom\'{e}e \cite{tho06} studied the convergence property \tn{of} the elliptic equation, and proposed two schemes: the symmetric scheme, \tn{in which} the optimal-order error \tn{is estimated} in the weighted $L^2$ norm, and the nonsymmetric scheme, \tn{in which} the $L^\infty$ error \tn{is estimated}. 
However, their finite element schemes \tn{are not easily adaptable to the} semilinear equation \eqref{eq:1}, as reported in our earlier study \cite{ns20}. \tn{Our earlier results are} briefly summarized below: 

\begin{itemize}
 \item If $f$ is \emph{globally} Lipschitz continuous, the solution of the symmetric scheme converges to the solution of \eqref{eq:1} in the weighted $L^2$ norm \tn{in} space and in the $L^\infty$ norm \tn{in} time. Moreover, the convergence is at the optimal order (see \tn{Theorem 4.1 in \cite{ns20}}). 
 \item If $f$ is \emph{locally} Lipschitz continuous and $N\le 3$, the same conclusion holds (see \tn{Theorem 4.3 in \cite{ns20}}). However, if $N\ge 4$, the convergence properties are not \tn{guaranteed}. \tn{For this reason, interest in} radially symmetric problems \tn{has diminished}. 
 \item If $f(u)=u|u|^\alpha$ with $\alpha\ge 1$ and the time partition is uniform, 
the solution of the non-symmetric scheme converges to the solution of \eqref{eq:1} in the $L^\infty(0,T;L^\infty(I))$ norm. \tn{Optimal-order convergence holds} up to the \tn{logarithmic} factor (see \tn{Theorem 4.6 in \cite{ns20}}). 
Nakagawa's time-increment control strategy \tn{is difficult to apply in such cases}. 
\end{itemize}

\tn{As the} non-symmetric scheme seems to be \tn{incompatible with} Nakagwa's time-increment control strategy\tn{, we pose} the following question:
Can the restriction $N\le 3$ \tn{be removed from} the symmetric scheme?
\tn{In} fact, this restriction \tn{is imposed by the} inverse inequality \tn{Lemma 4.8 in \cite{ns20}} and \tn{the necessity of finding the} boundedness of the finite element solution (see the proof of \tn{Theorem 4.3 in \cite{ns20}}). 
To surmount this difficulty, the $L^\infty$ estimates for the FEM \tn{can be directly derived using the discrete maximum principle (DMP)}. \tn{As the} DMP is based largely on the nonnegativity of the finite element solution\tn{, the time derivative term should be approximated by the} mass-lumping approximation. 
Unfortunately, we tried \tn{but} failed \tn{to prove} the convergence property of the finite element solution \tn{by this approximation} (see \eqref{eq:lump0} below). Therefore, we propose a new  mass-lumping approximation \eqref{eq:lump} in this paper. Using \tn{the new approximation}, we prove the DMP and \tn{the} convergence property of the finite element solution\tn{, and perform the} blow up analysis for any $N\ge 2$. 

\tn{Our typical results are summarized below. Here, our} schemes are \tn{denoted as}~(ML--1) and (ML--2). 
\begin{itemize}
 \item The solution of (ML--1) is non-negative if $f$ and $u^0$ \tn{satisfy} some conditions (Theorem \ref{prop:1}). 
Furthermore, if the time increment satisfies condition \eqref{eq:tau1}, the solution of (ML--2) is also non-negative (Theorem \ref{prop:1}). \tn{Theorem \ref{th:tau} gives a} useful sufficient condition \tn{of} \eqref{eq:tau1}. 
 \item The
solution of (ML--1) converges to the solution of \eqref{eq:1} in the
weighted $L^2$ norm \tn{in} space and in the $L^\infty$ norm \tn{in} time.
Moreover, the convergence is at the optimal order (Theorem \ref{th:ml1-1}). 
The proof is based on a sub-optimal estimate in the $L^\infty(0,T;L^\infty(I))$ norm (Theorem \ref{th:ml1-2}). 
 \item If condition \eqref{eq:tau1} is satisfied, then the solution of (ML--2) converges to the solution of \eqref{eq:1} in the $L^\infty(0,T;L^\infty(I))$ norm (Theorem \ref{th:ml2-1}). Unfortunately, the order of the convergence is sub-optimal. 
 \item The solution of (ML--2) \tn{reproduces} the blow up property of the solution of \eqref{eq:1} (Theorems \ref{prop:5.6} and \ref{prop:5.7}).  
\end{itemize}

This paper comprises six sections \tn{and} an Appendix. Section \ref{sec:2} presents our finite element schemes and the convergence theorems (\tn{Theorems} \ref{th:ml1-1}--\ref{th:ml2-1}). After \tn{describing our} preliminary results in Section \ref{sec:pre}, we prove \tn{our} convergence theorems in Section \ref{sec:42}. 
\tn{Section \ref{sec:5} reports the results of our} blow-up analysis\tn{, and} 
Section \ref{sec:7} \tn{validates our theoretical results with} numerical examples. 
Appendix \ref{sec:6} \tn{presents the proofs} of some auxiliary results on the eigenvalue problems. 

\section{The schemes and \tn{their} convergence results}
\label{sec:2}

Throughout this paper, $f$ is assumed \tn{as} a locally Lipschitz continuous function of $\mathbb{R}\to\mathbb{R}$. 

\tn{For some arbitrary} $\chi\in \dot{H}^{1}=\{v\in H^{1}(I)\mid v(1)=0\}$, \tn{we multiply} both sides of \eqref{eq:1a} by $x^{N-1}\chi$ and \tn{integrate} by parts over $I$. \tn{We thus} obtain 
\begin{equation}
\label{eq:w1a}
 \int_I x^{N-1} u_t\chi~dx+
 \int_I x^{N-1} u_{x}\chi_{x}~dx=
 \int_I x^{N-1} f(u)\chi~dx.
\end{equation}
Therefore a weak formulation of \eqref{eq:1} is stated as follows. 
For $t>0$, find $u(\cdot,t)\in \dot{H}^{1}$ such that 
\begin{equation}
\label{eq:w1}
(u_t,\chi)+A(u,\chi)=(f(u),\chi)\quad (\forall \chi\in \dot{H}^{1}),
\end{equation}
where 
\begin{equation} 
 \label{eq:4}
 (w,v)=\int_I x^{N-1} wv~dx, \quad 
A(w,v)=\int_I x^{N-1}w_{x}v_{x}~dx. 
\end{equation}

We now introduce the finite element method.  
For a positive integer $m$, we introduce node points 
\[
 0=x_0<x_1<\cdots<x_{j-1}<x_{j}<\cdots<x_{m-1}<x_m=1,
\]
and set $I_{j}=(x_{j-1},x_{j})$ and $h_j=x_j-x_{j-1}$, where $j=1,\ldots,m$. The granularity parameter is defined as $h=\max_{1\le j\le m}h_j$. 
Let $\mathcal{P}_k(J)$ be the set of all polynomials in an interval $J$ of degree $\le k$.
The $\mathcal{P}_1$ finite element space is defined as 
\begin{equation}
\label{eq:2}
S_{h}=\{ v \in H^{1}(I) \mid v\in\mathcal{P}_1(I_j)~(j = 1,\cdots,m),\  v(1)=0\}.
\end{equation}
The standard basis function $\phi_{j}\in S_h$, $j=0,1,\cdots,m$ is defined as 
\[
 \phi_{j}(x_{i})=\delta_{ij},
\]
where $\delta_{ij}$ denotes Kronecker's delta. We note that $S_h\subset \dot{H}^{1}$ and that any function of $\dot{H}^{1}$ is identified with a continuous function. The Lagrange interpolation operator $\Pi_{h}$ of $C(\overline{I})\to S_h$ is defined as $\Pi_{h}{w}=\dps\sum_{j=0}^{m}{w}(x_{j})\phi_{j}$ for ${w}\in C(\overline{I})$, where $C(\overline{I})$ denotes the set of all continuous functions in $\bar{I}$.   

The mass-lumping \tn{approximation of} the weighted $L^2$ norm $(\cdot,\cdot)$ \tn{can be naturally defined as}
\begin{equation}
\label{eq:lump0}
 (w,v)\approx 
w(x_0)v(x_0) \int_{0}^{x_{1/2}}x^{N-1}~dx 
+\sum_{i=1}^{m-1}w(x_i)v(x_i)\int_{x_{i-1/2}}^{x_{i+1/2}}x^{N-1}~dx,
\end{equation}
where $x_{i-1/2}=(x_i+x_{i-1})/2$. 
As mentioned in the Introduction, this standard \tn{formulation} is useless for our purpose. Instead, we \tn{define} 
\begin{equation}
\label{eq:lump} 
\dual{w,v}=\sum_{i=0}^{m-1}w(x_{i})v(x_{i})(1,\phi_{i})\qquad (w,v\in \dot{H}^{1}). 
\end{equation}
\tn{This definition leads to the} following result, which \tn{can be} verified by direct calculation. 

\begin{lemma}
\label{prop:dual}
We have $\dual{1,{w}}=(1,\Pi_{h}{w})$ for any ${w} \in C(\overline{I})$. 
\end{lemma}   


The associated norms with $(\cdot,\cdot)$ and $\dual{\cdot,\cdot}$ are \tn{respectively given by} 
\[
 \|v\|=(v,v)^{1/2}\quad \mbox{and}\quad \vnorm{v}=\dual{v,v}^{1/2}.
\]
\tn{These} norms are equivalent in $S_h$ as mentioned in Lemma \ref{prop:norm}.

\tn{The} time discretization \tn{is non-uniformly partitioned as}
\[
 t_0=0,\qquad t_{n} = \sum_{j=0}^{n-1}\tau_{j}\quad (n\ge 1),
\]
where $\tau_j>0$ denotes the time increment. Furthermore, \tn{we} set 
\[
 \tau=\sup_{j\ge 0}\tau_j.
\]


\tn{In general}, we write $\partial_{\tau_n}u_h^{n+1}=(u_h^{n+1}-u_h^n)/\tau_n$. \\
The finite element schemes \tn{are then stated as follows}.

\smallskip


\noindent \textbf{(ML--1)} Find $u_{h}^{n+1} \in S_{h}$, $n=0,1,\ldots$, such that  
\begin{equation}
\label{eq:3}
\dual{\partial_{\tau_n}u_h^{n+1},\chi} + A(u_{h}^{n+1},\chi)=(f(u_{h}^{n}),\chi)\quad 
(\chi \in S_{h}),
\end{equation} 
where $u_h^0\in S_h$ is assumed to be given. 

\medskip


\noindent \textbf{(ML--2)} Find $u_{h}^{n+1} \in S_{h}$, $n=0,1,\ldots$, such that  
\begin{equation}
\label{eq:8}
\dual{\partial_{\tau_n}u_h^{n+1},\chi} + A(u_{h}^{n},\chi)=\dual{f(u_{h}^{n}),\chi}\quad 
(\chi \in S_{h}).
\end{equation}

\tn{Below, we will show} the optimal order error estimate in the weighted $L^2$ norm for the solution of (ML--1). On the other hand, we are able to show only a sub-optimal order error estimate in the $L^\infty$ norm for the solution of (ML--2). 
Nevertheless, we consider (ML--2) because it is suitable for the blow-up analysis \tn{(see Section \ref{sec:5})}. 


\medskip

We \tn{also} summarize the well-posedness of our schemes. The proof is \tn{omitted because it is identical to Theorems 3.1 and 3.2 in}~\cite{ns20}. 

\begin{thm}
\label{prop:1}
Suppose that $n\ge 0$ and $u_h^n\in S_h$ are given.  
\begin{enumerate}
\item[\textup{(i)}] \tn{Schemes} \textup{(ML--1)} and \textup{(ML--2)} admit unique solutions $u_h^{n+1}\in S_h$. 
\item[\textup{(ii)}] In addition to the basic assumption on $f$, assume that 
$f$ is a non-decreasing function with $f(0)\ge 0$. 
If $u_{h}^n\ge 0$,
then the solution $u_h^{n+1}$ of \textup{(ML--1)} satisfies $u_h^{n+1}\ge 0$. 
\item[\textup{(iii)}] Under the assumptions of \textup{(ii)} above, \tn{further assume} that
\begin{equation} 
\label{eq:tau1}
\tau\le \min_{0\le i\le m-1}\frac{(1,\phi_{i})}{A(\phi_{i},\phi_{i})}.
\end{equation}
Then the solution $u_h^{n+1}$ of \textup{(ML--2)} satisfies $u_h^{n+1}\ge 0$.
\end{enumerate}
\end{thm}

\tn{To provide a} useful sufficient condition \tn{under which} \eqref{eq:tau1} \tn{holds}, we assume that 
the partition $\{x_i\}_{j=0}^m$ of $\bar{I}=[0,1]$ is quasi-uniform, that is, 
\begin{equation}
\label{eq:beta}
h \le \beta \min_{1\le j \le m}h_{j},
\end{equation}
where $\beta$ is a positive constant independent of $h$. 

\begin{thm}
\label{th:tau}
\tn{Inequality} \eqref{eq:tau1} holds if 
\begin{equation}
 \label{eq:tau1a}
\tau \le\frac{\beta^{2}}{N+1}h^{2}.
\end{equation}
\end{thm}

\begin{proof}
\tn{A} direct calculation \tn{gives} 
\[
\frac{(1,\phi_{i})}{A(\phi_{i},\phi_{i})}\ge \frac{1}{N+1}h_{i}^{2}\ge\frac{\beta^{2}}{N+1}h^{2}
\]
for any $i$. 
Therefore, \eqref{eq:tau1a} implies \eqref{eq:tau1}. 
\end{proof}

We \tn{now} proceed to the convergence analysis. 
Our results for (ML--1) and (ML--2) \tn{assume a smooth} solution $u$ of \eqref{eq:1}: given $T>0$ and setting $Q_T=[0,1]\times [0,T]$, we assume that $u$ is sufficiently smooth such that 
\begin{equation}
\label{eq:smooth1}
\kappa(u)=
\sum_{k=0}^{2} \|\partial_x^k u\|_{L^{\infty}(Q_T)}
+\sum_{l=1}^{2} \|\partial_t^l u\|_{L^{\infty}(Q_T)}
+\sum_{k=1}^{2}\|\partial_t\partial_x^ku\|_{L^{\infty}(Q_T)}
<\infty.
\end{equation}
\tn{Here, we have used the conventional} $\|v\|_{L^\infty(\omega)}=\max_{\overline{\omega}}|v|$ for a continuous function $v$ defined in a bounded set $\omega$ in $\mathbb{R}^p$, $p\ge 1$. 

Moreover, the approximate initial value $u_h^0$ is chosen as 
\begin{equation}
\label{eq:iv}
\|u_h^0-u^0\|_{L^{\infty}(I)}  \le C_0h^2 
\end{equation}
for a positive constant $C_0$. 

We \tn{now} express positive constants $C=C(\gamma_1,\gamma_2,\ldots)$ depending only on the parameters $\gamma_1,\gamma_2,\ldots$. Particularly, $C$ is independent of $h$ and $\tau$.

\begin{thm}[Optimal $L^2$ error estimate for (ML-1)]
\label{th:ml1-1} 
Assume that, for $T>0$, the solution $u$ of \eqref{eq:1} is sufficiently smooth that \eqref{eq:smooth1} holds. Moreover, assume that \eqref{eq:beta} and \eqref{eq:iv} are satisfied. 
Then, for sufficiently small $h$ and $\tau$, we have
\begin{equation}
\label{eq:ml1a} 
\sup_{0\le t_n\le T}\|u_{h}^{n}-u(\cdot,t_{n})\|\le C(h^{2}+\tau),
\end{equation}
where $C=C(T, f, \kappa(u), C_0, N,\beta)$ and $u_h^n$ is the solution of \textup{(ML--1)}.  
\end{thm}

The following result, which is worth \tn{a separate mention}, gives only a sub-optimal error estimate \tn{but is} useful \tn{for proving} Theorem \ref{th:ml1-1}. 

\begin{thm}[Sub-optimal $L^\infty$ error estimate for (ML--1)]
\label{th:ml1-2}
Under the assumptions of Theorem  \ref{th:ml1-1} \tn{and} for sufficiently small $h$ and $\tau$, \tn{we have} 
\begin{equation}
\label{eq:ml1b} 
\sup_{0\le t_n\le T}\|u_{h}^{n}-u(\cdot,t_{n})\|_{L^{\infty}(I)}
\le C\left(h+\tau\right),
\end{equation}
where $C=C(T, f, \kappa(u), C_0, N,\beta)$ and $u_h^n$ is the solution of \textup{(ML--1)}.  
\end{thm}  


\begin{thm}[Sub-optimal $L^\infty$ error estimate for (ML--2)]
\label{th:ml2-1}
\tn{Also under} the assumptions of Theorem \ref{th:ml1-1}, assume that \eqref {eq:tau1} is satisfied.  
Then, for sufficiently small $h$ and $\tau$, we have 
\begin{equation}
\label{eq:ml2-2}
 \sup_{0\le t_n\le T}\|u_{h}^{n}-u(\cdot,t_{n})\|_{L^{\infty}(I)} \le C\left(h+\tau\right),
\end{equation}
where $C=C(T, f, \kappa(u), C_0, N,\beta)$ and $u_h^n$ is the solution of \textup{(ML--2)}.   
\end{thm}                                                                                                                
\begin{rem}
Other schemes \tn{based on} the mass-lumping $\dual{\cdot,\cdot}$ are possible. For example, \tn{the scheme} 
\begin{equation}
\label{eq:8}
\dual{\partial_{\tau_n}u_h^{n+1},\chi} + A(u_{h}^{n+1},\chi)=\dual{f(u_{h}^{n}),\chi}\quad 
(\chi \in S_{h})
\end{equation} 
and 
\begin{equation}
\label{eq:8}
\dual{\partial_{\tau_n}u_h^{n+1},\chi} + A(u_{h}^{n},\chi)=(f(u_{h}^{n}),\chi)\quad 
(\chi \in S_{h}).
\end{equation} 
have very similar properties \tn{to} those of \textup{(ML--1)} and \textup{(ML--2)}. 
We \tn{omit the} details because the \tn{modifications are easily performed}.   
\end{rem}

\section{Preliminaries}
\label{sec:pre}

\tn{This section gives some} preliminary results \tn{of the theorem proofs.} 
\tn{The} quasi-uniformity condition \eqref{eq:beta} is always assumed. 

\tn{For some $w\in \dot{H}^1$,} the projection operator $P_A$ of $\dot{H}^1\to S_h$ associated with $A(\cdot,\cdot)$ \tn{is} defined as 
\begin{equation}
P_Aw\in S_h,\quad A(P_{A}w -w,\chi)=0\qquad (\chi\in S_{h}) .
 \label{eq:pA}
\end{equation}

\tn{The following error estimates are proved in} \cite{et84} and \cite{jes78}.
\begin{lemma}
\label{prop:tj}
Letting $w\in C^{2}(\bar{I})\cap\dot{H}^1$,  
\tn{and $h$ be sufficient small}, we obtain 
\begin{subequations} 
\begin{align}
\|P_Aw-w\| &\le Ch^{2}\|w_{xx}\|,  \label{eq:tj1} \\
\|P_Aw-w\|_{L^{\infty}(I)}&\le C\left( \log \frac{1}{h} \right)h^{2}\|w_{xx}\|_{L^{\infty}(I)}, \label{eq:tj2}
\end{align}
 \end{subequations}
where $C=(\beta,N)>0$. 
\end{lemma}


\begin{lemma}[Inverse estimate] 
\label{prop:inverse}
There \tn{exists a constant $C=C(\beta,N)>0$} such that
\[
\|{w}_x\|\le C h^{-1}\|{w}\|\qquad ({w}\in S_{h}).
\]
\end{lemma}

\begin{proof}
\tn{The proof is identical to that of the} standard inverse estimate.  
\end{proof}

\begin{lemma}
\label{prop:interval}
Let ${w}\in C(\overline{I})$ be a piecewise quadratic function in $I$, that is, ${w}|_{I_{j}}\in \mathcal{P}_{2}(I_{j})$ $(j=1,\cdots,m)$.
Then, we have 
\begin{equation}
\label{eq:interval1}
\int_{I}x^{N-1}|\Pi_{h}{w} -{w}|~dx \le Ch^{2}\|x^{N-1}{w}_{xx}\|_{L^{1}(I)},
\end{equation}
where $C=C(\beta,N)>0$.
\end{lemma}

\begin{proof}
\tn{To} prove \eqref{eq:interval1}, \tn{it suffices to replace} $I=I_j$ \tn{with} $j=1,\ldots,m$. First, let $j\ge 2$. By Taylor's theorem, we \tn{can} write
\begin{equation}
\label{eq:ek1}
 |\Pi_h{w}(x)-{w}(x)|\le Ch_j\int_{I_j}|{w}_{xx}(\xi)|~d\xi\quad (x\in I_j).
\end{equation}
\tn{Referring to} \eqref{eq:beta}, we see \tn{that} 
\begin{equation}
\label{eq:beta00}
\frac{x}{x_{j-1}}=1+\frac{x-x_{j-1}}{x_{j-1}}\le 1+\frac{h_{j}}{h_{j-1}}\le 1+\beta \quad (x\in \bar{I}_j).
\end{equation}
Combining \tn{\eqref{eq:ek1} and \eqref{eq:beta00}}, we deduce \tn{that}
\[
x^{N-1} |\Pi_h{w}(x)-{w}(x)|\le C(1+\beta)^{N-1}h_j\int_{I_j}x^{N-1}|{w}_{xx}(\xi)|~d\xi\quad (x\in {I}_j).
\]
Integrating both sides, we obtain \eqref{eq:interval1} for $I=I_j$ and $j\ge 2$. We \tn{now} proceed to the case $m=1$. Setting ${w}(x) = ax^{2}+bx +c$ for $x \in I_{1}$, \tn{where $a$, $b$ and $c$ are constants}, we express 
$\Pi_{h}{w}(x)-{w}(x)=ax(h_{1}-x)$ for $x \in I_{1}$. Therefore, we \tn{directly obtain} 
\begin{align*}
 \int_{I_1} x^{N-1}|\Pi_h{w}(x)-{w}(x)|~dx&=\frac{|a|}{(N+1)(N+2)}h_{1}^{N+2},\\ 
h_{1}^{2}\|x^{N-1}{w}_{xx}\|_{L^{1}(I_{1})}&=\frac{2}{N}|a|h_{1}^{N+2},
\end{align*}
which \tn{implies} \eqref{eq:interval1} for $I=I_1$. 
\end{proof}

\begin{lemma} 
\label{prop:norm}
There exist constants $C=C(\beta,N)$ and $C'=C'(\beta,N)$~such that
\[
C'\vnorm{{w}}\le \|{w}\|\le C\vnorm{{w}}\qquad ({w}\in S_{h}).
\]
\end{lemma}

\begin{proof}
Let ${w}\in S_h$. 
\tn{To} prove the first inequality, \tn{we note that} ${w}^2$ is a piecewise quadratic function and $({w}^2)_{xx}=2({w}_x)^2$.
\tn{By} Lemmas \ref{prop:interval} and \ref{prop:inverse}, \tn{we get}
\begin{align*}
 |\dual{{w},{w}}-({w},{w})|
&\le Ch^2\|x^{N-1}({w}^2)_{xx}\|_{L^1(I)}\\
&\le Ch^2\|{w}_x\|^2\\
&\le C\|{w}\|^2,
\end{align*}
which implies the first inequality. 

\tn{To} prove the second inequality, \tn{we} estimate $\|{w}\|$ as 
\begin{align*}
\|{w}\|^{2} &\le \sum_{j=1}^{m}\left[ {w}(x_{j-1})^{2}\int_{x_{j-1}}^{x_{j}}x^{N-1}dx + {w}(x_{j})^{2}\int_{x_{j-1}}^{x_{j}}x^{N-1}dx\right]\\
 &= {w}(x_{0})^{2}\int_{x_{0}}^{x_{1}}x^{N-1}dx + \sum_{j=1}^{m-1}{w}(x_{j})^{2}\int_{x_{j-1}}^{x_{j+1}}x^{N-1}dx.
\end{align*}
\tn{We also} express $\vnorm{{w}}$ as
\[
\vnorm{{w}}^{2}={w}(x_{0})^{2}\int_{x_{0}}^{x_{1}}x^{N-1}\phi_{0}(x)~dx + \sum_{j=1}^{m-1}{w}(x_{j})^{2}\int_{x_{j-1}}^{x_{j+1}}x^{N-1}\phi_{j}(x)~dx.
\]

Therefore, it suffices to show that
\begin{subequations} 
\begin{align}
\label{eq:x0}
\int_{x_{0}}^{x_{1}}x^{N-1}dx 
&\le C_{1}\int_{x_{0}}^{x_{1}}x^{N-1}\phi_{0}~dx, \\
\label{eq:xj}
\int_{x_{j-1}}^{x_{j+1}}x^{N-1}dx 
&\le C_{2}\int_{x_{j-1}}^{x_{j+1}}x^{N-1}\phi_{j}~dx\quad (j=1,\ldots,m),
\end{align} 
\end{subequations}
where $C_{1}=C_{1}(N)>0$ and $C_{2}=C_{2}(N)>0$.    
 
Equation \eqref{eq:xj} is \tn{directly} verified using \eqref{eq:beta00}. 
Equation \eqref{eq:x0} is obtained \tn{by the change-of-variables technique}, setting $\xi=x/h_1$.    
\end{proof}

We here introduce \tn{two} auxiliary problems. 
Given $n\ge 0$, $g_h^n\in S_h$ and $u_h^n\in S_h$, we \tn{seek} $u_{h}^{n+1} \in S_{h}$ such that 
\begin{equation}
\label{eq:ml1} 
\dual{\partial_{\tau_n}u_h^{n+1},\chi} + A(u_{h}^{n+1},\chi)=\dual{g_h^n,\chi}\qquad 
(\chi \in S_{h}),
\end{equation} 
and 
\begin{equation}
\label{eq:ml2} 
\dual{\partial_{\tau_n}u_h^{n+1},\chi} + A(u_{h}^{n},\chi)=\dual{g_h^n,\chi}\qquad 
(\chi \in S_{h}).
\end{equation} 

\begin{lemma}
 \label{prop:comp1} 
Suppose that $n\ge 0$ and $u^n_h,g_h^n\in S_h$ are given.
Then, \tn{problem} \eqref{eq:ml1} admits a unique solution $u_h^{n+1}\in S_h$ and it satisfies 
\begin{equation}
 \label{eq:infest}
\|u_h^{n+1}\|_{L^\infty(I)}\le 
\|u_h^n\|_{L^\infty(I)}+ \tau_n \|g_h^n\|_{L^\infty(I)}.  
\end{equation}
\tn{Problem} \eqref{eq:ml2} \tn{also} admits a unique solution $u_h^{n+1}\in S_h$ \tn{that satisfies} \eqref{eq:infest} under \tn{condition} \eqref{eq:tau1}.  
\end{lemma}

\begin{proof}
The unique existence of \tn{the} solution of \eqref{eq:ml1} \tn{can be} verified \tn{by a} standard \tn{approach (see Theorems 3.1, 3.2 in \cite{ns20})}. 
Substituting $\chi=\phi_i$, $i=0,\ldots,m-1$, \tn{in} \eqref{eq:ml1}, we have 
\[
  \frac{\tau_n a_{i,i-1}}{m_i}u_{i-1}^{n+1}
+\left(1+\frac{\tau_n a_{i,i}}{m_i}\right)u_{i}^{n+1}
+\frac{\tau_n a_{i,i+1}}{m_i}u_{i+1}^{n+1}
=u_i^n+\tau_ng_i^n,
\]
where  $u_i^n=u_h^n(x_i)$, $g_i^n=g_h^n(x_i)$, $a_{i,j}=A(\phi_j,\phi_i)$ and $m_i=(1,\phi_i)$. 
Therein, we should understand \tn{that} $a_{0,-1}=0$, $m_0=1$ and $u_{-1}^{n+1}=1$. 
Moreover, 
substituting $\chi=\phi_i$ \tn{in} \eqref{eq:ml2}, we get 
\[
  u_i^{n+1}=
-\frac{\tau_n a_{i,i-1}}{m_i}u_{i-1}^{n}
+\left(1-\frac{\tau_n a_{i,i}}{m_i}\right)u_{i}^{n}
-\frac{\tau_n a_{i,i+1}}{m_i}u_{i+1}^{n}
+\tau_ng_i^n.
\]
\tn{From these} expressions, \eqref{eq:infest} \tn{is deduced by a} standard argument. 
%
\end{proof}

\section{Proofs of Theorems \ref{th:ml1-1}, \ref{th:ml1-2} and \ref{th:ml2-1}}
\label{sec:42}

\begin{proof}[Proof of Theorem \ref{th:ml1-1} using Theorem \ref{th:ml1-2}]
\tn{This proof is} divided into \tn{the following} two steps:

\noindent \emph{Step 1.}We prove Theorem \ref{th:ml1-1} under \tn{an} additional assumption\tn{:} $f$ is a globally Lipschitz function. That is, we assume
\begin{equation}
\label{eq:f1}
M=\sup_{\begin{subarray}{c}s,s'\in\mathbb{R} \\ s\ne s'\end{subarray}}\frac{|f(s)-f(s')|}{|s-s'|}<\infty.
\end{equation}
Using $P_Au$, we divide the error into the form 
\begin{equation}
\label{eq:pa}
u_{h}^{n}-u(\cdot,t_{n})=\underbrace{(u_{h}^{n}- P_{A}u(\cdot,t_{n}))}_{=v_{h}^{n}} + \underbrace{(P_{A}u(\cdot,t_{n}) - u(\cdot,t_{n}))}_{=w^{n}}. 
\end{equation}
\tn{From \eqref{eq:tj1},} we know that
\begin{equation}
\label{eq:th1.12}
\|w^{n}\| \le Ch^{2}\|u_{xx}(t_n)\| \le Ch^2 \|u_{xx}\|_{L^\infty(Q_T)}
\end{equation}
and that $\partial_{\tau_{n}}P_{A}v = P_{A}\partial_{\tau_{n}}v$ for $v\in C(\overline{I})$.

We \tn{now estimate} $v_{h}^n$. 
Using the weak form \eqref{eq:w1} at $t=t_{n+1}$, \tn{scheme} (ML--1), and the property of $P_{A}$, we obtain 
\begin{equation}
 \dual{\partial_{\tau_{n}}v_{h}^{n+1},\chi} + A(v_{h}^{n+1},\chi)=
(\textup{I}+\textup{II}+\textup{III}+\textup{IV}+\textup{V})(\chi)\quad (\chi\in S_h),
\label{eq:tmp1}
\end{equation}
where
\begin{align*}
\textup{I}(\chi) &= (f(u_{h}^{n}),\chi)-(f(u(\cdot,t_{n})),\chi),\\
\textup{II}(\chi)  &= (u_{t}(t_{n+1}),\chi)-(\partial_{\tau_{n}}u(\cdot,t_{n+1}),\chi),\\
\textup{III}(\chi)  &= (f(u(\cdot,t_{n})),\chi) - (f(u(\cdot,t_{n+1})),\chi),\\
\textup{IV}(\chi) &= (\partial_{\tau_{n}}u(\cdot,t_{n+1}),\chi)-(P_{A}\partial_{\tau_{n}}u(\cdot,t_{n+1}),\chi)=(\partial_{\tau_{n}}w^{n+1},\chi),\\
\textup{V} (\chi)& = (\partial_{\tau_{n}}P_{A}u(\cdot,t_{n+1}),\chi) - \dual{\partial_{\tau_{n}}P_{A}u(\cdot,t_{n+1}),\chi} .
\end{align*}
\tn{The estimations of} $\textup{I}$--$\textup{IV}$ are straightforward. That is, we have 
\begin{align*}
 |\textup{I}(\chi)| & \le  M(\|w^{n}\|+ \|v_{h}^{n}\|)\cdot\|\chi\|,\\
 |\textup{II}(\chi)| & \le \tau_{n} \|u_{tt}\|_{L^{\infty}(Q_T)}  \|\chi \|,\\
 |\textup{III}(\chi)| & \le  \tau_{n}M\|u_{t}\|_{L^{\infty}(Q_T)}\|\chi\|,\\
 |\textup{IV}(\chi)| & \le Ch^{2}\|u_{txx}\|_{L^{\infty}(Q_T)}\|\chi\|. 
\end{align*}

To \tn{estimate} $\textup{V}$, we use Lemmas 
\ref{prop:dual} and \ref{prop:interval}. 
\tn{Lemma \ref{prop:interval} is applicable because} $\partial_{\tau_{n}}P_{A}u(\cdot,t_{n+1})\chi$ is a piecewise quadratic function. 
That is, 
\begin{align*}
|\textup{V}(\chi)|  &=(1, \Pi_{h} (\partial_{\tau_{n}}P_{A}u(\cdot,t_{n+1})\chi)-\partial_{\tau_{n}}P_{A}u(\cdot,t_{n+1})\chi) \\
&\le  Ch^{2}\|x^{N-1}\{\partial_{\tau_{n}}P_{A}u(\cdot,t_{n+1})\chi\}_{xx}\|_{L^{1}(I)}\\
&\le Ch^{2}\|x^{N-1}(P_{A}\partial_{\tau_{n}}u(\cdot,t_{n+1}))_x\cdot \chi_x\|_{L^{1}(I)}\\
&\le Ch^{2}\|(P_{A}\partial_{\tau_{n}}u(\cdot,t_{n+1}))_x\|\cdot\|\chi_x\|\\
&\le Ch^{2}\|(\partial_{\tau_{n}}u(\cdot,t_{n+1}))_x\|\cdot\|\chi_x\|\\
&\le Ch^{2}\|u_{tx}\|_{L^\infty(Q_T)}\|\chi_x\|.
\end{align*}

\tn{Substituting} $\chi=v_{h}^{n+1}$ \tn{in} \eqref{eq:tmp1} gives 
\begin{multline*}
 \frac{1}{2\tau_n}\left(
\vnorm{v_{h}^{n+1}}^2-\vnorm{v_{h}^{n}}^2
\right)+\|(v_{h}^{n+1})_x\|^2 \le  \\
C\vnorm{v_{h}^n}\cdot \|(v_{h}^{n+1})_x\| + 
C(h^2+\tau_n)\kappa(u)\|(v_{h}^{n+1})_x\| .
\end{multline*}
\tn{Herein,} we have used Lemma \ref{prop:norm} and the Poincar{\' e} inequality (\tn{Lemma 18.1 in \cite{tho06}}). \tn{By} Young's inequality, \tn{we then deduce that} 
\begin{equation*}
\frac{1}{\tau_n}\left(
\vnorm{v_{h}^{n+1}}^2-\vnorm{v_{h}^{n}}^2
\right)\le 
C\vnorm{v_{h}^{n}}^2+C(h^2+\tau_n)^2\kappa(u)^2.
\end{equation*}
Therefore, 
\[
 \vnorm{v_{h}^{n}}^2\le e^{CT}\vnorm{v_{h}^{0}}^2+C(\exp(CT)-1)(h^2+\tau_n)^2\kappa(u)^2,
\]
which completes the proof.  

\smallskip

\noindent \emph{Step 2.} 
Let ${r}=1+\|u\|_{L^\infty(Q_T)}$. 
Consider \eqref{eq:1} and (ML--1) with replacement $f(s)$ in:
\[
\tilde{f}(s)=
\begin{cases}
f(r)& (s\ge {r})\\ 
f(s) & (|s|\le{r})\\
-f(-r) & (s\le -{r}). 
\end{cases}
\]
The function $\tilde{f}$ is a locally Lipschitz function \tn{satisfying} 
\[
 M
=
\sup_{\begin{subarray}{c}s,s'\in\mathbb{R} \\ s\ne s'\end{subarray}}
\frac{|\tilde{f}(s)-\tilde{f}(s')|}{|s-s'|}
=
\sup_{\begin{subarray}{c}|s|,|s'|\le r \\ s\ne s'\end{subarray}}
\frac{|{f}(s)-{f}(s')|}{|s-s'|}.
\]
Let $\tilde{u}$ and $\tilde{u}_{h}^{n}$ be the solutions of \eqref{eq:1} and (ML-1) with $\tilde{f}$, respectively.
Applying Step 1 and Theorem \ref{th:ml1-2} to $\tilde{u}$ and $\tilde{u}_{h}^{n}$, 
we obtain 
\begin{subequations} 
 \begin{align}
\sup_{0\le t_n\le T}\|\tilde{u}_{h}^{n}-\tilde{u}(\cdot,t_{n})\| &\le C(h^{2}+\tau),\label{eq:ml1.1} \\
\sup_{0\le t_n\le T}\|\tilde{u}_{h}^{n}-\tilde{u}(\cdot,t_{n})\|_{L^{\infty}(I)}& \le C\left(h+h^{2}\log\frac{1}{h}+\tau\right).
\label{eq:ml1.2}
\end{align}
\end{subequations}
By the definition of ${r}$ and the uniqueness of the solution of \eqref{eq:1}, we know \tn{that} $u=\tilde{u}$ in $Q_T$. \tn{For} sufficiently small $h$ and $\tau$, \tn{we have} 
\[
C\left(h+h^{2}\log\frac{1}{h}+\tau\right)\le 1.
\]
Consequently, $\|\tilde{u}_{h}^{n}\|_{L^{\infty}(I)}\le {r}$ for $0\le t_n\le T$ and, by the uniqueness of \tn{the} solution of (ML--1),  we have $u_h^n=\tilde{u}_h^n$. 
Therefore, \eqref{eq:ml1.1} implies the desired \tn{result}.  
\end{proof}


We \tn{now} proceed to the proof of Theorem \ref{th:ml1-2}.

\begin{proof}[Proof of Theorem \ref{th:ml1-2}]
\tn{The} notation is \tn{that of} the previous proof. 
It suffices to prove Theorem \ref{th:ml1-2} under assumption \eqref{eq:f1}\tn{, which is generalizable to an arbitrary} $f$ as \tn{demonstrate in} the previous proof.  
\tn{By} \eqref{eq:tj2}, \tn{we have} 
$\|v_h^0\|_{L^\infty(I)}\le Ch^2\log(1/h)\kappa(u)$ and 
$\|w^n\|_{L^\infty(I)}\le Ch^2\log(1/h)\kappa(u)$ for $0\le t_n\le T$.  
Therefore, it remains to \tn{estimate} $v_h^n$ when $0< t_n\le T$. 
Setting 
\[
G_h^n=\sum_{i=0}^{m-1}G_i^n\phi_i,\quad 
 G_i^n=\frac{(\textup{I}+\textup{II}+\textup{III}+\textup{IV}+\textup{V})(\phi_i)}{(1,\phi_i)},
\]
we rewrite \eqref{eq:tmp1} as 
\begin{equation*}
 \dual{\partial_{\tau_{n}}v_{h}^{n+1},\chi} + A(v_{h}^{n+1},\chi)=
\dual{G_h^n,\chi}\quad (\chi\in S_h).
\end{equation*}

\tn{Showing that} 
\begin{equation}
 \|G_h^n\|_{L^\infty(I)}\le 
M\|v_h^n\|_{L^\infty(I)}+
C\left(h+\tau\right)\kappa(u),
\label{eq:G}
\end{equation}
we can apply Lemma \ref{prop:comp1} to obtain 
\[
 \|v_h^{n+1}\|_{L^\infty(I)}\le (1+M\tau_n)\|v_h^{n}\|_{L^\infty(I)} + \tau_n\cdot C(h+\tau)\kappa(u),
\] 
and, consequently, 
\[
 \|v_h^n\|_{L^\infty(I)} \le e^{Mt_n} \|v_h^0\|_{L^\infty(I)} 
+\frac{e^{Mt_n}-1}{M}C(h+\tau)\kappa(u).
\]
\tn{Thereby}, we deduce the desired estimate. 

Below we prove \tn{the truth of} \eqref{eq:G}. 
Recall that we \tn{assumed} global Lipschitz continuity \eqref{eq:f1} on $f$. 

$\textup{I}(\phi_i)$--$\textup{IV}(\phi_i)$ are \tn{straightforwardly estimated as follows:} 
\begin{align*}
 |\textup{I}(\phi_i)| & \le  M[\|v_{h}^{n}\|_{L^\infty(I)}+Ch^2\log(1/h)\kappa(u)]\cdot(1,\phi_i),\\
 |\textup{II}(\phi_i)| & \le \tau_{n} \kappa(u) (1,\phi_i),\\
 |\textup{III}(\phi_i)| & \le  M \tau_{n} \kappa(u) (1,\phi_i),\\
 |\textup{IV}(\phi_i)| & \le Ch^{2}\log(1/h) \kappa(u) (1,\phi_i). 
\end{align*}

To \tn{estimate} $\textup{V}(\phi_i)$, we write
\[
  \textup{V}(\phi_i)=\textup{V}_1(\phi_i)+\textup{V}_2(\phi_i)+\textup{V}_3(\phi_i),
\]
where 
\begin{align*}
\textup{V}_1(\phi_i)&=(\partial_{\tau_{n}}P_{A}u(\cdot,t_{n+1}),\phi_i) -(\partial_{\tau_{n}}u(\cdot,t_{n+1}),\phi_i), \\
\textup{V}_2(\phi_i)&=(\partial_{\tau_{n}}u(\cdot,t_{n+1}),\phi_i) - \dual{\partial_{\tau_{n}}u(\cdot,t_{n+1}),\phi_i},\\
\textup{V}_3(\phi_i)&=\dual{\partial_{\tau_{n}}u(\cdot,t_{n+1}),\phi_i}- \dual{\partial_{\tau_{n}}P_{A}u(\cdot,t_{n+1}),\phi_i}.
\end{align*}
\tn{The above terms are respectively estimated} as: 
\begin{align*}
 |\textup{V}_1(\phi_i)|
&\le \| P_{A}(\partial_{\tau_{n}}u(\cdot,t_{n+1}))-\partial_{\tau_{n}}u(\cdot,t_{n+1})\|_\infty (1,\phi_i)\\
&\le Ch^2\log(1/h)\kappa(u) (1,\phi_i);\\
 |\textup{V}_2(\phi_i)|
&\le 
\int_{I} x^{N-1}\left|\partial_{\tau_{n}}u(x,t_{n+1})-\partial_{\tau_{n}}u(x_i,t_{n+1})\right| \phi_i(x)~dx\\
&\le Ch\kappa(u) (1,\phi_i);\\
|\textup{V}_3(\phi_i)|
&\le | \partial_{\tau_{n}}u(x_i,t_{n+1})-P_{A}\partial_{\tau_{n}}u(x_i,t_{n+1})| (1,\phi_i)\\
&\le Ch^2\log(1/h)\kappa(u) (1,\phi_i).
\end{align*}
\tn{We thereby} deduce \tn{that} 
\[
 \|G_h^n\|_{L^\infty(I)}\le M\|v_h\|_{L^\infty(I)}+C\left(h+h^2\log(1/h)+\tau\right)\kappa(u)
\]
which implies \eqref{eq:G}. This \tn{step} completes the proof. 
\end{proof}

\begin{proof}[Proof of Theorem \ref{th:ml2-1}]
\tn{The proof is identical to} that of Theorem \ref{th:ml1-2}. 
\end{proof}


\section{Blow-up analysis}
\label{sec:5}

\subsection{Results}

\tn{This} section \tn{considers the} spacial nonlinearlity
\[
 f(s)=s|s|^{\alpha},\quad \alpha>0.
\]
\tn{As we are} interested in non-negative \tn{solutions}, we assume that 
\begin{equation}
\label{eq:ivp}
 u^0\ge 0,\not\equiv 0,\quad u_h^0\ge 0,\not\equiv 0.
\end{equation}
\tn{Therefore}, the solution $u$ of \eqref{eq:1} is non-negative and the solution $u_h^n$ of {(ML--2)} is also non-negative under condition \eqref{eq:tau1}. 
Generally, the solution of \eqref{eq:1} blows up \tn{when the} initial data $u_0$ \tn{are sufficiently large}, and the blow up is controlled by the energy functional associated with \eqref{eq:1}. 
Herein, we study whether or not the numerical solution behaves similarly by initially defining some properties of the solution $u$ of \eqref{eq:1}.
In particular, we see that {(ML--2)} is suitable for this purpose. 

\tn{The} energy functionals associated with \eqref{eq:1} \tn{are defined} as 
\begin{align*}
K(v)&=\frac{1}{2}\|v_{x}\|^{2}-\frac{1}{\alpha+2}\int_{I}x^{N-1}|v(x)|^{\alpha+2}~dx,\\
I(v)&=\int_{I}x^{N-1}v(x)\psi(x)~dx,
\end{align*}
where $\psi\in\dot{H}^{1}$ denotes the eigenfunction associated with the first eigenvalue $\mu>0$ 
of the eigenvalue problem
\begin{equation}
 \label{eq:evp}
A(\psi,\chi)=\mu(\psi,\chi)\quad (\chi\in \dot{H}^{1}). 
\end{equation}
Without loss of generality, we assume that~$\psi\ge0~\mbox{in}~I~\mbox{and}~\int_{I}x^{N-1}\psi(x)~dx=1$.

\tn{The following propositions} \ref{prop:5.1}, \ref{prop:5.2}, and \ref{prop:5.3} are \tn{often applied to the} semilinear heat equation in a bounded domain. \tn{They are easily extended to the} radially symmetric case.  

\begin{prop}
$K(u(t))$ is a non-increasing function of $t$, where $u$ is the solution of \eqref{eq:1}.  
\label{prop:5.1}
\end{prop}   

\begin{prop}
\label{prop:5.2}
Suppose that $u^0\ge 0, \not\equiv 0$ and $u$ is the solution of \eqref{eq:1}. Then, the following statements are equivalent: 
\begin{itemize}
 \item[\textup{(i)}] There exists $T_{\infty}>0$ such that $u$ blows up at $t=T_{\infty}$ in the sense that $\displaystyle{\lim_{t\to T_{\infty}}\|u(\cdot,t)\|=\infty}$.
 \item[\textup{(ii)}] There exists $t_{0}\ge 0$ such that $K(u(\cdot,t_{0}))<0$.
\end{itemize}
\end{prop}

\begin{prop}
\label{prop:5.3}
Suppose that $u^0\ge 0, \not\equiv 0$ and that $u$ is the solution of \eqref{eq:1}. Then, the following statements are equivalent: 
\begin{itemize}
 \item[\textup{(i)}] There exists  $T_{\infty}>0$ such that $u$ blows up at $t=T_{\infty}$ in the sense that 
$\displaystyle{\lim_{t\to T_{\infty}}I(u(\cdot,t))=\infty}$.
 \item[\textup{(ii)}] There exists $t_{0}\ge 0$ such that $I(u(\cdot,t_{0}))>\mu^{\frac{1}{\alpha}}$.
\end{itemize}
\end{prop}

\begin{rem}
In Propositions \ref{prop:5.2} and \ref{prop:5.3},  the blow up time $T_{\infty}$ is estimated \tn{respectively as}
\[
T_{\infty}\le t_{0}+\frac{\alpha+2}{\alpha^{2}} N^{-\frac{\alpha}{2}}\|u(\cdot,t_0)\|^{-\alpha}.
\] 
and 
\[
T_{\infty}\le t_{0}+\int_{I(u(\cdot,t_{0}))}^{\infty}\frac{ds}{-\mu s+s^{1+\alpha}}.
\] 
\end{rem}

\medskip

We \tn{now} proceed to the discrete energy functionals. To this end, we \tn{employ} the finite element version of the eigenvalue problem: 
\begin{equation}
\label{eigen_h}
A(\hat{\psi}_{h},\chi_{h})=\hat{\mu}_{h}\dual{\hat{\psi}_{h},\chi_{h}}\quad (\chi_{h}\in S_{h}).
\end{equation}
\tn{Let} $\hat{\psi}_{h}\in S_{h}$ be the eigenfunction associated with 
the smallest eigenvalue $\hat{\mu}_{h}>0$ of \eqref{eigen_h}. 
For the eigenvalue problem \eqref{eigen_h}, we \tn{state} the following result\tn{, postponing} the proof \tn{to the} Appendix.  

\begin{prop}
\label{conv:eigenvalue-main}
If the partition $\{x_j\}_{j=0}^m$ is quasi-uniform, that is, \tn{satisfies} \eqref{eq:beta}, we have the following\tn{:}
\begin{itemize}
 \item[\textup{(i)}] $\hat{\mu}_{h}\to \mu$ as $h\to 0$.
 \item[\textup{(ii)}] The first eigenfunction $\hat{\psi}_{h}$ of \eqref{eigen_h} does not change sign.
 \item[\textup{(iii)}] $\|(\hat{\psi}_{h}-\psi)_x\|\to 0$ as $h\to 0$. 
\end{itemize}
\end{prop}



\medskip

Therefore, \tn{without loss of generality,} we \tn{can} assume that $\hat{\psi}_{h}\ge 0$ and $\int_{I}x^{N-1}\hat{\psi}_{h}(x)~dx=1$. 

\tn{For} $v\in S_h$, \tn{we set}  
\begin{align*}
K_{h}(v)&=\frac{1}{2}\|v_{x}\|^{2}-\frac{1}{\alpha+2}\sum_{i=0}^{m}|v(x_{i})|^{\alpha+2}(1,\phi_{i}),\\
I_{h}(v)&=\dual{v,\hat{\psi}_{h}}=\int_{I}x^{N-1}\Pi_{h}(v\hat{\psi}_{h})(x)~dx.
\end{align*}

We introduce the approximate blow-up time~$\hat{T}_{\infty}(h)$ by setting 
\begin{equation}
\hat{T}_{\infty}(h)=\lim_{n\to\infty}t_{n}=\lim_{n\to\infty}\sum_{j=0}^{n-1}\tau_{j}.
\end{equation}

We are now in a position to mention the main theorems in this section:

\begin{thm}
\label{prop:5.6}
Let \eqref{eq:ivp} be satisfied. 
Suppose that the solution $u$ of \eqref{eq:1} blows up at finite time $T_{\infty}$ in the sense that 
\begin{equation}
\label{blow1}
\|u(\cdot,t)\|_{L^{\infty}(I)}\to \infty \  \mbox{ and } \  \|u(\cdot,t)\|_{L^2(I)}\to \infty\quad (t\to T_{\infty}-0).
\end{equation}
Assume that for any $T<T_\infty$, $u$ is sufficiently smooth that \eqref{eq:smooth1} holds. 
Assuming \tn{also} that \eqref{eq:beta} is satisfied, we set 
\begin{equation}
 \label{eq:tau9}
\tau=\delta \frac{\beta^2}{N+1}h^2
\end{equation}
for some $\delta\in (0,1]$. 
The time increment $\tau_{n}$ is \tn{iteratively} defined as 
\begin{equation}
\label{time1}
\tau_{n}=\tau_{n}(h)=\tau\min\left\{1,\frac{1}{\vnorm{u_{h}^{n}}^{\alpha}}\right\}
\end{equation}
\tn{where we have used} the solution $u_h^n$ of \textup{(ML--2)}. 
Moreover, assume that \eqref{eq:tau1} is satisfied and that
\begin{equation}
\label{eq:K}
\forall T< T_{\infty},\quad 
\lim_{h\to0}\sup_{0\le t_{n}\le T}|K(u(\cdot,t_{n}))-K_{h}(u_{h}^{n})|=0.
\end{equation}
\tn{We then} have 
\begin{equation}
\lim_{h\to0}\hat{T}_{\infty}(h)=T_{\infty}.
 \label{eq:77}
\end{equation}
\end{thm}

\begin{thm}
\label{prop:5.7}
Let \eqref{eq:ivp} be satisfied. 
Suppose that the solution $u$ of \eqref{eq:1} blows up at finite time $T_{\infty}$ in the sense that 
\begin{equation}
\label{blow2}
I(u(\cdot,t))\to\infty\quad \mbox{and}\quad \|u(\cdot,t)\|_{L^{\infty}(I)}\to \infty\quad (t\to T_{\infty}-0).
\end{equation}
Assume that, for any $T<T_\infty$, $u$ is sufficiently smooth that \eqref{eq:smooth1} holds. 
Assuming \tn{also} that \eqref{eq:beta} is satisfied, we set $\tau$ by \eqref{eq:tau9} with some $\delta\in (0,1]$. 
The time increment $\tau_{n}$ is \tn{iteratively} defined as 
\begin{equation}
\label{time2}
\tau_{n}=\tau_{n}(h)=\tau\min\left\{1,\frac{1}{I_{h}(u_{h}^{n})^{\alpha}}\right\}
\end{equation}
\tn{where we have used} the solution $u_h^n$ of \textup{(ML--2)} with \eqref{eq:iv}. \tn{We then} obtain \eqref{eq:77}. 
\end{thm}

\begin{rem}
\tn{The above theorems differ in that Theorem \ref{prop:5.6} requires the} convergence property \eqref{eq:K} of the discrete energy functional $K_{h}(u_{h}^{n})$, whereas \tn{no} convergence property \tn{of} $I_h$ is necessary in Theorem \ref{prop:5.7}.  
%
\end{rem}

\begin{rem}
Unfortunately, we could not prove Theorems \ref{prop:5.6} and \ref{prop:5.7} using the solution of (ML--1). In particular, the proof of the difference inequalities \eqref{eq:J10} and \eqref{eq:h2} \tn{failed in scheme (ML--1)}. 
\end{rem}

\subsection{Proof of Theorem \ref{prop:5.6}}

\tn{To prove Theorem \ref{prop:5.6}, we} follow Nakagawa's blow-up analysis \cite{nak76}. 
\tn{For this purpose}, \tn{we must} derive 
the difference inequality \eqref{eq:J10} and the boundedness \eqref{eq:T10} of $\hat{T}_\infty$ \tn{(see Lemmas \ref{prop:5.4} and \ref{prop:5.5})}. \tn{The original proof in \cite{nak76} immediately follows from these} results; 
see also \cite{che86}, \cite{che92}, and \cite{cho07}. Therefore, we concentrate our \tn{efforts on proving} Lemmas \ref{prop:5.4} and \ref{prop:5.5}. 

Throughout this \tn{subsection}, we \tn{take} the same assumptions \tn{of} Theorem \ref{prop:5.6}\tn{; in} particular, the time-increment control \eqref{time1}. \tn{Note} that \tn{condition} \eqref{eq:tau1} is satisfied by the definition of $\tau_n$. 
Consequently, the solution $u$ of \eqref{eq:1} and the solution $u_h^n$ of (ML--2) are non-negative. 

\begin{lemma}
\label{prop:5.3a}
$K_{h}(u_{h}^{n})$ is a non-\tn{increasing} sequence of $n$. 
\end{lemma}

\begin{proof}
\tn{Fixing some} $n\ge 0$, we write $w=u_h^{n+1}$, $u=u_h^n$, $w_j=w(x_j),$ and $u_j=u({x_j})$. 
\tn{To} show that $K_{h}(w)-K_{h}(u)\le 0$, we \tn{perform the following division}:
\[
 K_{h}(w)-K_{h}(u)=\textup{X}+\textup{Y},
\]
where 
\begin{align*}
\textup{X} &=\frac{1}{2}\|w_{x}\|^{2}-\frac{1}{2}\|u_{x}\|^{2},\\
\textup{Y} &= -\frac{1}{\alpha+2}\sum_{j=0}^{m-1}w_j^{\alpha+2}(1,\phi_{j})+\frac{1}{\alpha+2}\sum_{j=0}^{m-1}u_j^{\alpha+2}(1,\phi_{j}).
\end{align*}
$\textup{X}$ is expressed as 
\[
 \textup{X} =A(u,w-u)+\frac{1}{2}A(w-u,w-u).
\] 
\tn{By} the mean value theorem, there exists $\theta_j\in [0,1]$ such that 
\[
w_j^{\alpha+2}-u_j^{\alpha+2}=(\alpha+2)\tilde{u}_j^{\alpha+1}(w_j-u_j),
\]
where $\tilde{u}_j={u}+\theta_{j}(w-u)$. Therefore, 
\begin{align*}
\textup{Y}& = 
-\sum_{j=0}^{m-1} \tilde{u}_{j}^{\alpha+1}(1,\phi_{j}) (w_j-u_j) \\
&=-\sum_{j=0}^{m-1}[ \tilde{u}_{j}^{\alpha+1}-u^{\alpha+1}] (1,\phi_{j}) (w_j-u_j)-\sum_{j=0}^{m-1}u_j^{\alpha+1}(1,\phi_{j})(w_j-u_j)=\textup{Y}_1+\textup{Y}_2.
\end{align*}

We calculate 
\[
 A(u,w-u)+\textup{Y}_2= -\dual{\frac{w-u}{\tau_n},w-u}=-\frac{1}{\tau}\vnorm{w-u}^2 
\]
and 
\[
  \textup{Y}_1 = -\sum_{j=0}^{m-1}(\alpha+1)\theta_j(1,\phi_j)\hat{u}_j^\alpha (w_j-u_j)^2\le 0,
\]
where $\hat{u}_j=u+\hat{\theta}_j(\tilde{u}_j-u_j)$ with some $\hat{\theta}_j\in [0,1]$. 

\tn{Meanwhile}, for $v_{h}\in S_{h}(I)$, \tn{we write}
\begin{align*}
A(v_{h},v_{h})
&\le 2\sum_{j=1}^{m}\int_{I_{j}}x^{N-1}\cdot\frac{1}{h_{j}^{2}}~dx\cdot(v_{j}^2+v_{j-1}^2)\\
&=2\sum_{j=0}^{m-1}a_{jj}v_{j}^{2}.
\end{align*}
Using \eqref{eq:tau1}, we have 
\[
 A(v_{h},v_{h})\le 2\sum_{j=0}^{m-1} \frac{ (1,\phi_j)}{\tau_n }v_j^2
=\frac{2}{\tau_{n}}\dual{v_{h},v_{h}}=\frac{2}{\tau_{n}}\vnorm{v_{h}}^2\quad (v_h\in S_h).
\]

\tn{We thereby deduce that}
\[
  \textup{X}+\textup{Y}
= -\frac{1}{\tau_n}\vnorm{w-u}^2+\frac{1}{2}A(w-u,w-u)+\textup{Y}_1\le 0,
\]
which implies that $K_h(u_h^n)$ is non-increasing in $n$. 
 \end{proof}

\begin{lemma}
\label{prop:5.4}
If there exists a non-negative integer $n'$ such that $K_{h}(u_{h}^{n})\le 0$ for all $n\ge n'$, we have 
\begin{equation}
 \frac{1}{2}\partial_{\tau_{n}}\vnorm{u_{h}^{n+1}}^{2}\ge \frac{\alpha}{\alpha+2}N^{\frac{\alpha}{2}}\vnorm{u_{h}^{n}}^{\alpha+2} \quad (n\ge n').
\label{eq:J10}
\end{equation}
\end{lemma}

\begin{proof}
Substituting $\chi_{h}=u_{h}^{n}$ \tn{in} \textup{(ML--2)}, we obtain
\[
\dual{\frac{u_{h}^{n+1}-u_{h}^{n}}{\tau_{n}},u_{h}^{n}}+A(u_{h}^{n},u_{h}^{n})=\dual{u_{h}^{n}(u_{h}^{n})^{\alpha},u_{h}^{n}}.
\]
We note that 
\[
\dual{u_{h}^{n+1}-u_{h}^{n},u_{h}^{n}}
\le \dual{u_{h}^{n+1}-u_{h}^{n},\frac{1}{2}(u_{h}^{n+1}+u_{h}^{n})}
=\frac{1}{2\tau_n}(\vnorm{u_h^{n+1}}^2-\vnorm{u_h^{n}}^2).
\]
\tn{By the} decreasing property of $K_{h}(u_{h}^{n})$, \tn{we have} 
\[
 \|(u_{h}^{n})_{x}\|^{2}\le\frac{2}{\alpha+2}\dual{(u_{h}^{n})^{\alpha+2},1}\quad (n\ge n').
\]
Combining \tn{these results}, we get 
\[
\frac{1}{2}\cdot\frac{1}{\tau_{n}}(\vnorm{u_{h}^{n+1}}^{2}-\vnorm{u_{h}^{n}}^{2})
\ge\frac{\alpha}{\alpha+2}\dual{(u_{h}^{n})^{\alpha+2},1}.
\]

Using H{\" o}lder's inequality, we calculate 
\[
 \vnorm{u_{h}^{n}}^{2}\le \left( 1/N\right)^{\frac{\alpha}{\alpha+2}}\cdot\dual{(u_{h}^{n})^{\alpha+2},1}^{\frac{2}{\alpha+2}}.
\]
\tn{We thereby} deduce \eqref{eq:J10}. 
\end{proof}

\begin{lemma}
\label{prop:5.5}
If $J_{h}(u_{h}^{n_{0}})\le 0$ and $\vnorm{u_{h}^{n_{0}}}\ge 1$ for some integer $n_0\ge 0$, then we have  
\begin{equation}
 \hat{T}_{\infty}(h)\le t_{n_{0}}+\left\{\frac{\alpha+2}{\alpha^{2}}N^{-\frac{\alpha}{2}}+\tau\left(  1+\frac{2}{\alpha}\right)\right\}\vnorm{u_{h}^{n_{0}}}^{-\alpha}.
\label{eq:T10}
\end{equation}
\end{lemma}

\begin{proof}
From Lemma \ref{prop:5.4},
\[
\vnorm{u_{h}^{n+1}}^{2}\ge(1+2\tau_{n}C\vnorm{u_{h}^{n}}^{\alpha})\vnorm{u_{h}^{n}}^{2}=(1+2\tau C)\vnorm{u_{h}^{n}}^{2},
\]
where $C=\frac{\alpha}{\alpha+2}N^{\frac{\alpha}{2}}$. Therefore, 
\[
\lim_{n\to\infty}\vnorm{u_{h}^{n}}=\infty
\]
and, for $n\ge n_{0}$,
\[
t_{n}=t_{n_{0}}+\sum_{m=n_{0}}^{n-1}\tau_{m}=t_{n_{0}}+\sum_{m=n_{0}}^{n-1}\frac{\tau}{\vnorm{u_{h}^{m}}^{\alpha}}.
\]
The \tn{remainder} is \tn{identical to} that of \tn{Corollary 2.1 in \cite{cho07}, so the details are omitted here}.   
\end{proof}

\subsection{Proof of Theorem \ref{prop:5.7}}
\tn{To prove Theorem \ref{prop:5.7}, we} apply abstract theory by \tn{(Propositions 4.2 and 4.3) in} \cite{ss16b}. 
\tn{In this subsection}, we \tn{take} the same assumptions \tn{of} Theorem \ref{prop:5.7}, \tn{in} particular, the time-increment control \eqref{time2}. 

\begin{lemma}
\label{lem5.9}
We have 
$\displaystyle{T_{\infty}\le \liminf_{h\to0}}~\hat{T}_{\infty}(h)$. 
\end{lemma}

\begin{proof}
\tn{The proof is shown} by contradiction. Setting $\displaystyle{S_{\infty}=\liminf_{h\to0}}~\hat{T}_{\infty}(h)$, 
we assume that $S_{\infty}<T_{\infty}$. 
Then, there exists $h_{0}>0$ such that $\hat{T}_{\infty}(h)<  M$ for all $h\le h_{0}$, 
where $M=\frac{T_{\infty}+S_{\infty}}{2}<T_{\infty}$.
That is, for \tn{some} fixed $h\le h_{0}$, we have $t_{n}\le \hat{T}_{\infty}(h)<M$ and $I_h(u_h^n)\to \infty$ as $n\to\infty$. This implies that \tn{for some $0\le j(n)\le m-1$\tn{, we have} $u_h^{n}(x_{j(n)})\to \infty$}, and consequently $\|u_{h}^{n}\|_{L^{\infty}(I)}\to \infty$. 
However, from Theorem~\ref{th:ml2-1}, we \tn{observe that}
\[
\lim_{h\to0}\sup_{0\le t_n\le M}\|u_{h}^{n}-u(\cdot,t_{n})\|_{L^{\infty}(I)}=0 .
\]
\tn{If this expression is true, then} $T_{\infty}$ \tn{cannot be the} blow-up time of the solution $u$ of \eqref{eq:1}. 
\tn{This contradiction completes the proof.}
\end{proof}

\begin{lemma}
\label{energy_error}
For any $T<T_{\infty}$, \tn{we have}
\[
\lim_{h\to0}\sup_{0\le t_{n}\le T}|I_{h}(u_{h}^{n})-I(u(\cdot,t_{n}))|=0.
\]
\end{lemma}

\begin{proof}
We derive \tn{separate} estimations for $|I_{h}(u_{h}^{n})-\tilde{I}_{h}(u_{h}^{n})|$ and $|\tilde{I}_{h}(u_{h}^{n})-I(u(\cdot,t_{n}))|$, where $\tilde{I}_{h}(v)$ denotes \tn{the} auxiliary functional
\[
\tilde{I}_{h}(v)=\int_{I}x^{N-1}v(x)\hat{\psi}_{h}(x)~dx.
\]
 
\tn{From} Theorem \ref{th:ml2-1}, Lemma \ref{prop:interval} and Proposition \ref{conv:eigenvalue-main} (iii), we \tn{first} derive  
\begin{align*}
|I_{h}(u_{h}^{n})-\tilde{I}_{h}(u_{h}^{n})|
&\le Ch^{2}\|x^{N-1}(u_{h}^{n} \hat{\psi}_{h})_{xx}\|_{L^{1}(I)} \\
&\le Ch^{2} \|(u_{h}^{n})_x\| \cdot \|\hat{\psi}'_{h}\| \\
& \le Ch\|u_{h}^{n}\|_{L^\infty(I)}\|\hat{\psi}'_{h}\| \\
& \le Ch (\|u(\cdot,t_n)\|_{L^\infty(I)}+1)(\|\psi'\|+1),
\end{align*}
where we have used \tn{the elemental} inequality 
$\|\psi_x\|\le C h^{-1}\|\psi\|_{L^{\infty}(I)}$ for $\psi\in S_{h}$.
This implies that 
$|I_{h}(u_{h}^{n})-\tilde{I}_{h}(u_{h}^{n})| \to 0$ as $h\to 0$.


\tn{Conversely}, as $h\to 0$, we have 
\begin{align*}
|\tilde{I}_{h}(u_{h}^{n})-I(u(\cdot,t_{n}))|
&=\left|\int_{I}x^{N-1}(u_{h}^{n}-u(\cdot,t_{n}))\hat{\psi}_{h}~dx\right|+\left|\int_{I}x^{N-1}u(\cdot,t_{n})(\hat{\psi}_{h}-\psi)~dx\right|\\
&\le \|u_{h}^{n}-u(\cdot,t_{n})\|\cdot\|\hat{\psi}_{h}\|+\|u(\cdot,t_{n})\|\cdot\|\hat{\psi}_{h}-\psi\| \to 0.
\end{align*}
\tn{This expression concludes the proof}. 
\end{proof}

The following is a readily obtainable consequence of Lemma \ref{energy_error}. 
\begin{lemma}
\label{lem:s0}
There exists a nonnegative integer $n_{0}=n_{0}(h)$ such that
$I_{h}(u_{h}^{n_{0}})>s_{0}$. 
\end{lemma}

The following lemma is elementary and \tn{was} originally stated \tn{as Lemma 3.1 in \cite{cho07}}. 

\begin{lemma}
\label{lem5.11}
There exists $s_0>1$ satisfying 
\[
\frac12 f(s)+(1+\mu) s\le f(s)\quad (s\ge s_{0}),
\]
where $f(s)=s^{1+\alpha}$. 
\end{lemma}

We can now \tn{prove Theorem \ref{prop:5.7}}. 

\begin{proof}[Proof of Theorem \ref{prop:5.7}]
It remains to verify \tn{that}
\begin{equation}
\label{eq:butime1}
T_{\infty}\ge \limsup_{h\to 0}~\hat{T}_{\infty}(h). 
\end{equation}
To this end, we apply abstract theory \tn{(Propositions 4.2 and 4.3) in \cite{ss16b}}. 
\tn{Adopting the} notation of \cite{ss16b}, \tn{we respectively set} $X$, $X_{h}$, $G$, $H$, $J$ and $J_{h}$ in \cite{ss16b} as 
\begin{align*}
&X=\dot{H}^{1},\quad \|v\|_{X}=\|v\|_{L^{\infty}(I)},\quad X_{h}=S_{h},\\
&G(s)=\frac{1}{2}f(s)=\frac{1}{2}s^{1+\alpha},\quad H(s)=s^{\alpha}\quad (s\ge 0),\\
&J(t,v)=\int_{I}x^{N-1}v(x)\cdot\psi(x)~dx,\quad (t,v)\in(0,\infty)\times X,\\
&J_{h}(t,v_{h})=I_{h}(v_{h})=\sum_{j=0}^{m-1}v_{j}\hat{\psi}_{h}(x_{j})(1,\phi_{j}),\quad (t,v_{h})\in(0,\infty)\times X_{h}, \quad ~v_{j}=v_{h}(x_{j}).
\end{align*}  
\tn{Here}, $G$ and $H$ are functions of class (G) and class (H), \tn{respectively}. 

\tn{To} avoid \tn{unnecessary} complexity, we assume \tn{that} $n_0=0$ \tn{(see Lemma \ref{lem:s0})}.  
Our problem setting \tn{matches} problem \tn{settings} (I)--(VIII) in \tn{\S 4 of \cite{ss16b}}. 

It is readily \tn{verified} that conditions (H1), (H3), (H4) and (H5) in \tn{\S 4 of \cite{ss16b}} hold. 
We \tn{need only} check \tn{that condition} (H2) 
\begin{equation}
 \label{eq:h2}  
\partial_{\tau_{n}}I_{h}(u_{h}^{n+1}) \ge \frac12 f(I_{h}(u_{h}^{n}))
\quad (n\ge 0)
 \end{equation}
\tn{also holds}.

Substituting $\chi_{h}=\hat{\psi}_{h}\in S_{h}$ \tn{in} \textup{(ML--2)} and using the relation \eqref{eigen_h}, we have 
\[
 \partial_{\tau_{n}}I_{h}(v_{h}^{n+1})+\hat{\mu}_{h}I_{h}(v_{h}^{n+1})=\dual{f(v_{h}^{n}),\hat{\psi}_{h}}.  \]

From Proposition \ref{conv:eigenvalue-main} (i), we know \tn{that} $\hat{\mu}_{h}<\mu+1$.
Moreover, because $\dual{1,\hat{\psi}_{h}}=1$, we can apply Jensen's inequality to get 
\[
\partial_{\tau_{n}}I_{h}(u_{h}^{n+1})\ge -(\mu+1)\cdot I_{h}(u_{h}^{n})+f(I_{h}(u_{h}^{n})).
\] 

\tn{By} Lemma \ref{lem5.11}, $-(\mu+1)s+f(s)\ge\frac{1}{2}f(s)$ for $s\ge s_0$.
Because $I_{h}(u_{h}^{0})>s_{0}$ by Lemma \ref{lem:s0}, we deduce \tn{that}
\[
\partial_{\tau_{n_{0}}}I_{h}(u_{h}^{1})\ge\frac{1}{2}f(I_{h}(u_{h}^{0})).
\] 
\tn{We thus obtain} $I_{h}(v_{h}^{1})>s_{0}$. \tn{By this process}, we \tn{finally} obtain  
\[
\partial_{\tau_{n}}I_{h}(u_{h}^{n+1})\ge\frac{1}{2}f(I_{h}(u_{h}^{n})),\quad 
I_{h}(u_{h}^{n})>s_{0}\quad (n\ge 0).
\]
\end{proof}

\begin{rem}
Theorem~\ref{prop:5.7} remains true \tn{after replacing} \eqref{time2} by 
\[
\tau_{n}=\tau_{n}(h)=\tau\min\left\{1,\frac{1}{\|u_{h}^{n}\|_{L^{\infty}(I)}^{\alpha}}\right\}.
\]
However, \tn{this definition increases} the computational cost \tn{over that based on}~\eqref{time2}. 
Therefore, \tn{in the following numerical evaluation, we adopt} \eqref{time2}.
\end{rem}

\section{Numerical examples}
\label{sec:7}

This section \tn{validates our theoretical results with} numerical examples. 

\tn{We first} examine the error estimates of the solutions \tn{on a} uniform spatial mesh $x_j=jh$ ($j=0,\ldots,m$) \tn{with} $h=1/m$, \tn{regarding} the numerical solution with $h'=1/480$ as the exact solution. 
The following quantities were compared: 
\begin{align*}
\mathcal{E}_1(h)&= \|u_{h'}^{n}-u_{h}^{n}\|_{L^{1}(I)},\\
\mathcal{E}_2(h)&=\left\|u_{h'}^{n}-u_{h}^{n}\right\|=\left\|x^{\frac{N-1}{2}}(u_{h'}^{n}-u_{h}^{n})\right\|_{L^{2}(I)},\\
 \mathcal{E}_\infty (h)&= \|u_{h'}^{n}-u_{h}^{n}\|_{L^{\infty}(I)}.
\end{align*}

Fig.~\ref{fig:3} shows the results for 
$N=3$, $\alpha=4$ and $u(0,x)=\cos\frac{\pi}{2}x$. 
\tn{The} time increment \tn{was uniform} ($\tau_n=\tau=\lambda h^2,$ $n=0,1,\ldots,$ $\lambda=1/2)$ and \tn{the iterations were continued} until $t\le T=0.005$. 
\tn{In scheme} (ML-1), the theoretical convergence rate \tn{was} $h^2+\tau$ in the $\|\cdot\|$ norm (see Theorem \ref{th:ml1-1}), \tn{but was slightly deteriorated} in the $L^\infty$ norm.

\begin{figure}[htbp]
        \begin{center}
           \includegraphics[width=.4\textwidth]{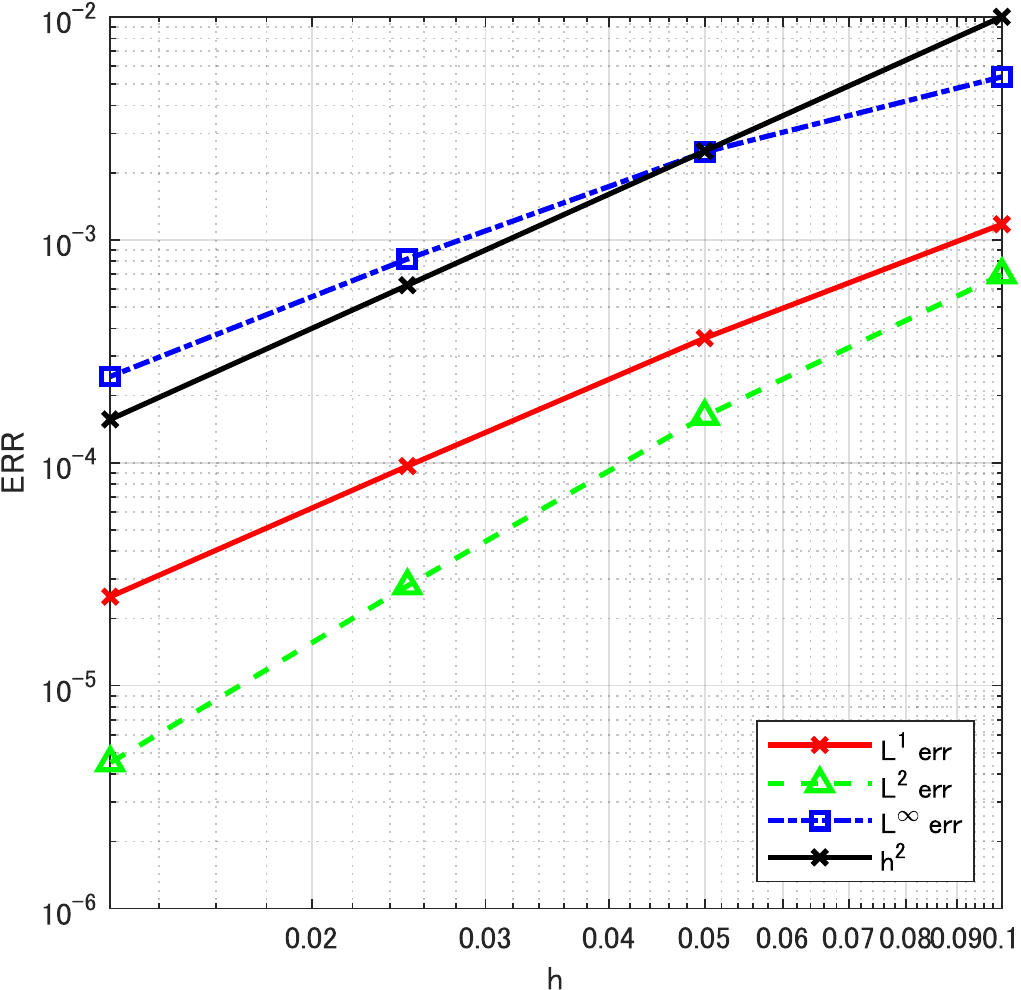}\\
        \end{center}
     \caption{\tn{Error convergences versus granularity: $N=3$, $\alpha=4$ and $u(0,x)=\cos\frac{\pi}{2}x$}}
\label{fig:3}
\end{figure}

\tn{In the case} $N=4$, which is not supported in the convergence property of \tn{the} standard symmetric FEM \cite{ns20}, \tn{we} chose $\alpha=3$ and $u(0,x)=3\cos\frac{\pi}{2}x$. 
\tn{The} errors \tn{were computed up to} $T=0.0033$ \tn{on} the non-uniform meshes \tn{with} $x_i=\sin\frac{(i-1)\pi}{2m}$ and $\tau_n$ with $\lambda=0.11$. 
\tn{As shown in} Fig.~\ref{fig:4}, the $\|\cdot\|$ norm \tn{showed second-order convergence} in \tn{both} (ML-1) \tn{and standard FEM}, \tn{but the $\|\cdot\|_{L^{\infty}(I)}$ norm showed} first-order convergence in the standard FEM.

\begin{figure}[htbp]
      \begin{minipage}{.49\textwidth}
        \begin{center}
           \includegraphics[width=.9\textwidth]{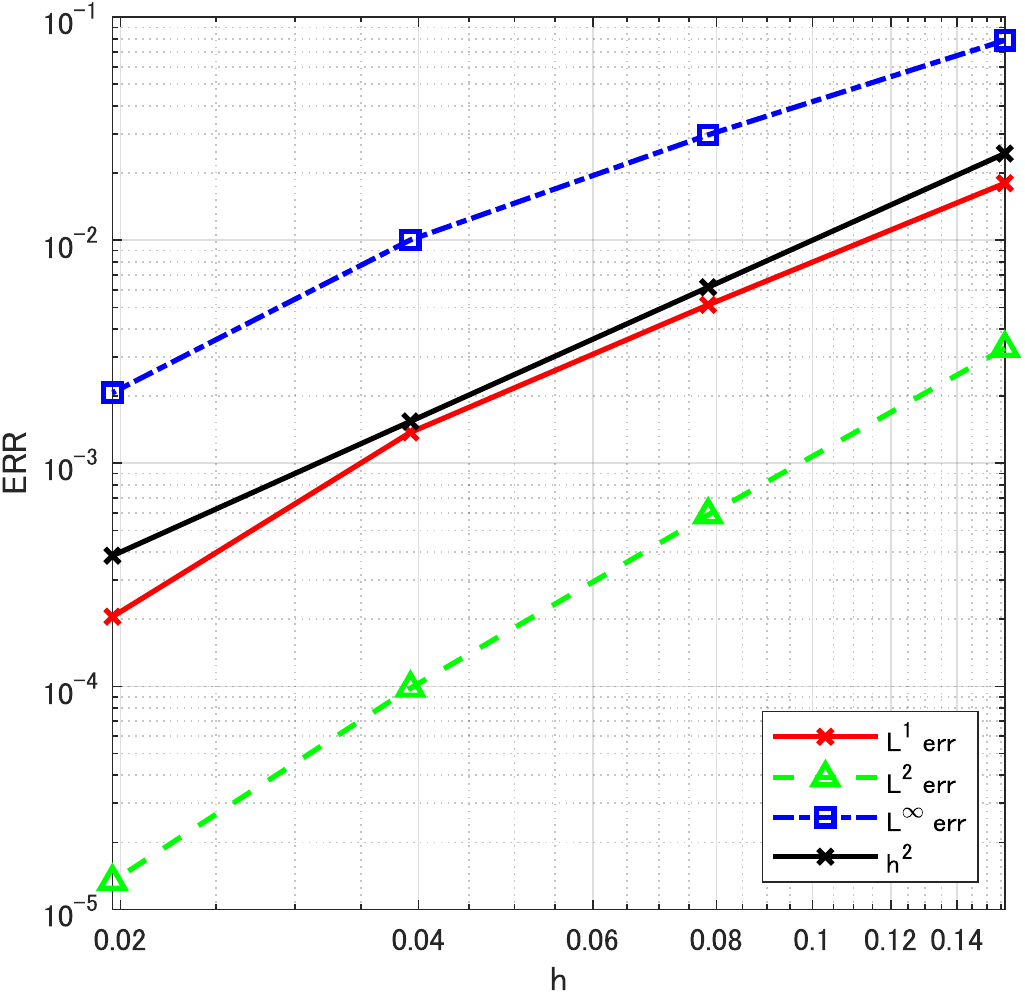} \\
         (a) \textup{(ML-1)}
        \end{center}
      \end{minipage}
      \begin{minipage}{.49\textwidth}
        \begin{center}
          \includegraphics[width=.9\textwidth]{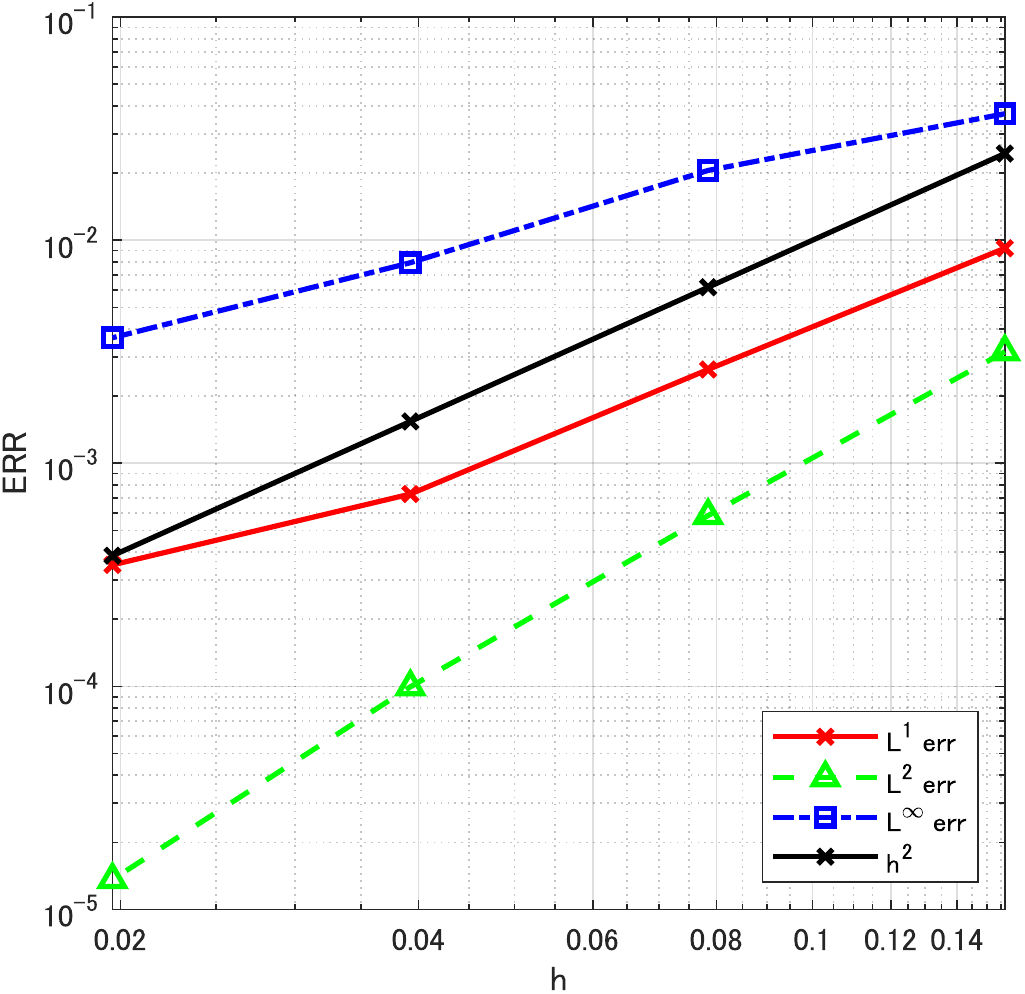}\\
          (b) Standard FEM
        \end{center}
      \end{minipage}
     \caption{\tn{Error convergences in the (ML--1) schemes (left) and the standard finite element method (right): $N=4$, $\alpha=3$ and $u(0,x)=3\cos\frac{\pi}{2}x$}}
\label{fig:4}
\end{figure}

 Second, we confirmed the non-increasing property of the energy functional $K_{h}(u_{h}^{n})$ \tn{in scheme} (ML--2) \tn{with} $N=5$, $\alpha=\frac{4}{3}$, and $u(0,x)=\cos\frac{\pi}{2}x,~13\cos\frac{\pi}{2}x$. \tn{The} time increment $\tau_{n}$ \tn{was determined} through Theorem~5.6 with $\beta=1$~and~$\delta=1$. \tn{Simulations were performed on a} uniform \tn{spatial} mesh $x_j=jh$ with $h=1/m$ and $m=50$. 
\tn{The} results \tn{(see Fig.~\ref{fig:5})} support Lemma~5.9. 
\tn{As shown in} Fig.~\ref{fig:6}, the energy functional $I_{h}(u_{h}^{n})$ \tn{increased} exponentially after~$t=0.04$~\tn{when the initial data was large}, \tn{but vanishes when the} initial data \tn{was small}. For $I_{h}(u_{h}^{n})$, we used the time increment in Theorem~5.7 with $\delta=1$.
\begin{figure}[htbp]
      \begin{minipage}{0.49\textwidth}
        \begin{center}
           \includegraphics[width=0.9\textwidth]{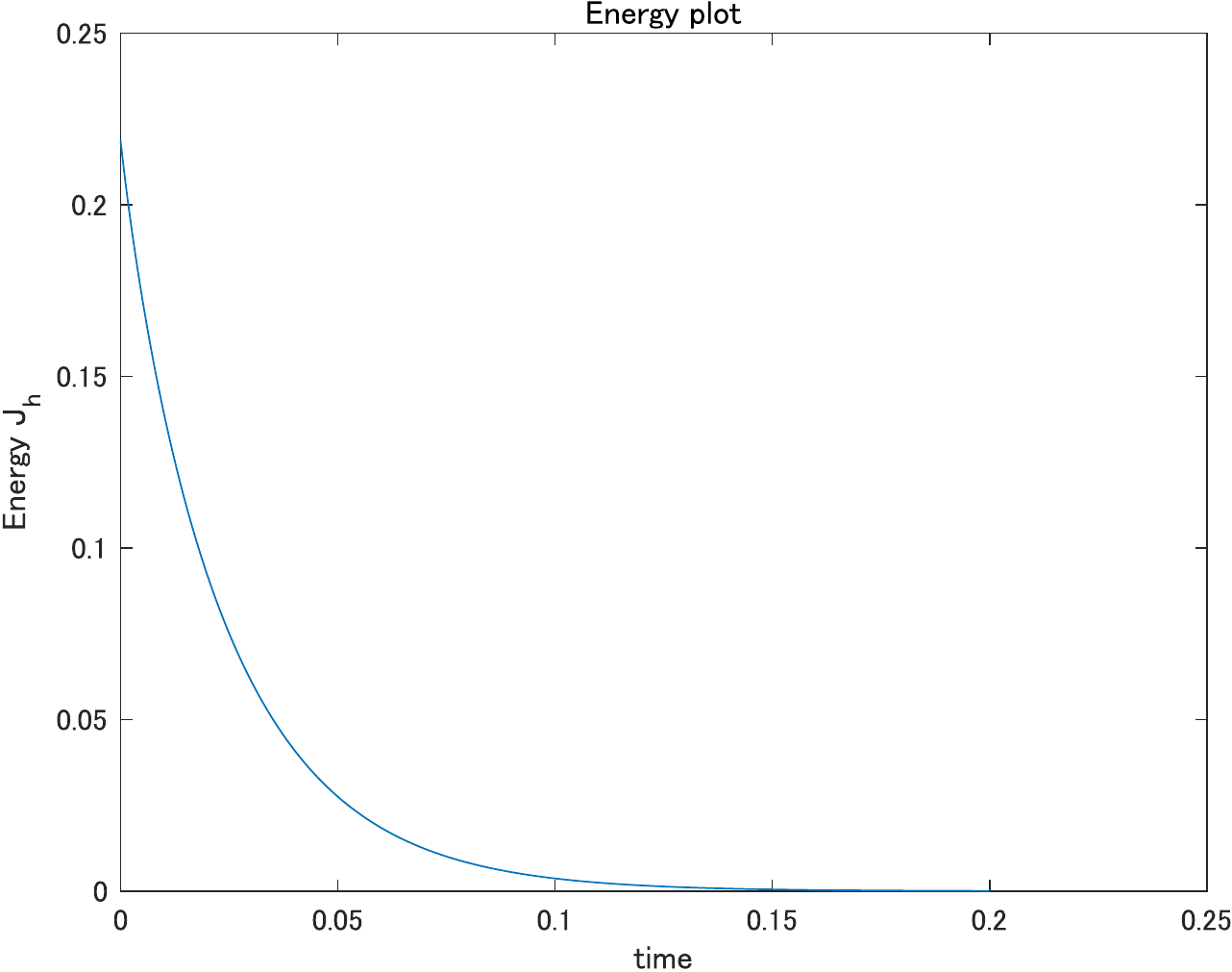} \\
         \textup{(ML--2)} \& $u(0,x)=\cos\frac{\pi}{2}x$
        \end{center}
      \end{minipage}
      \begin{minipage}{0.49\textwidth}
        \begin{center}
           \includegraphics[width=0.9\textwidth]{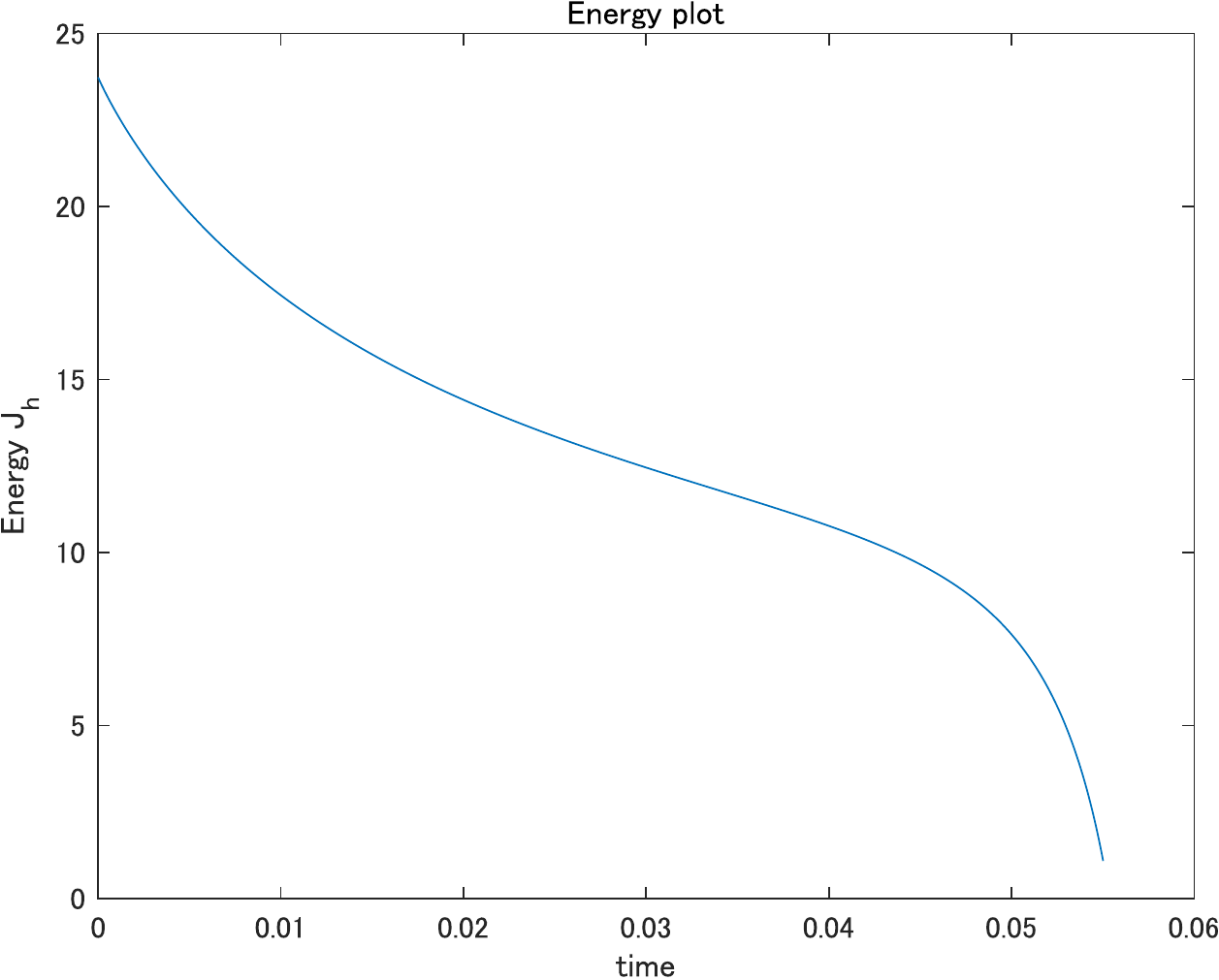}\\
         \textup{(ML--2)} \& $u(0,x)=13\cos\frac{\pi}{2}x$
        \end{center}
      \end{minipage}
     \caption{\tn{Temporal dynamics of the energy} functional $K_{h}(u_{h}^{n})$ \tn{in scheme (ML--2) with different initial conditions} }
\label{fig:5}
\end{figure}

\begin{figure}[htbp]
      \begin{minipage}{0.49\textwidth}
        \begin{center}
           \includegraphics[width=0.9\textwidth]{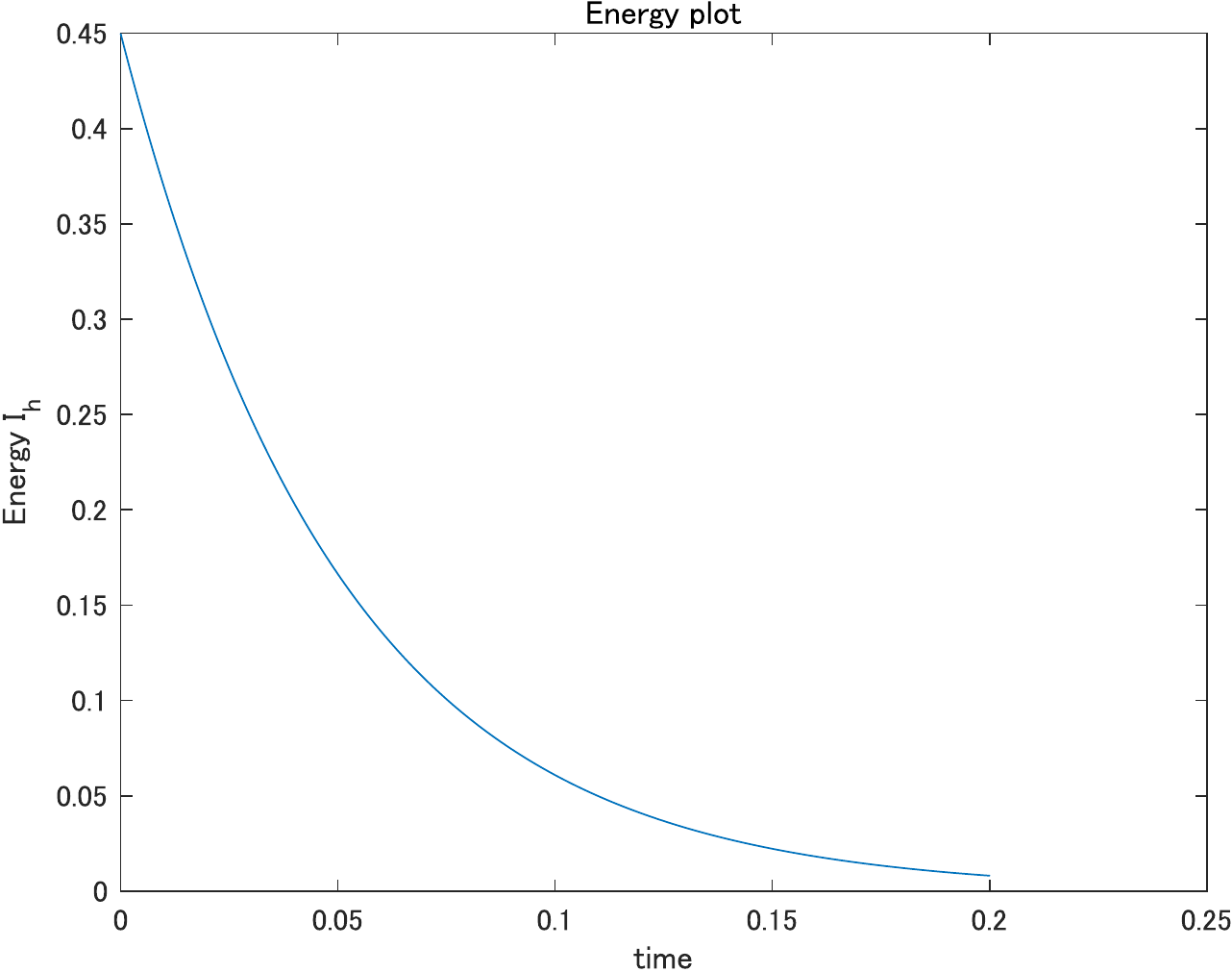} \\
         \textup{(ML--2)} \& $u(0,x)=\cos\frac{\pi}{2}x$
        \end{center}
      \end{minipage}
      \begin{minipage}{0.49\textwidth}
        \begin{center}
           \includegraphics[width=0.9\textwidth]{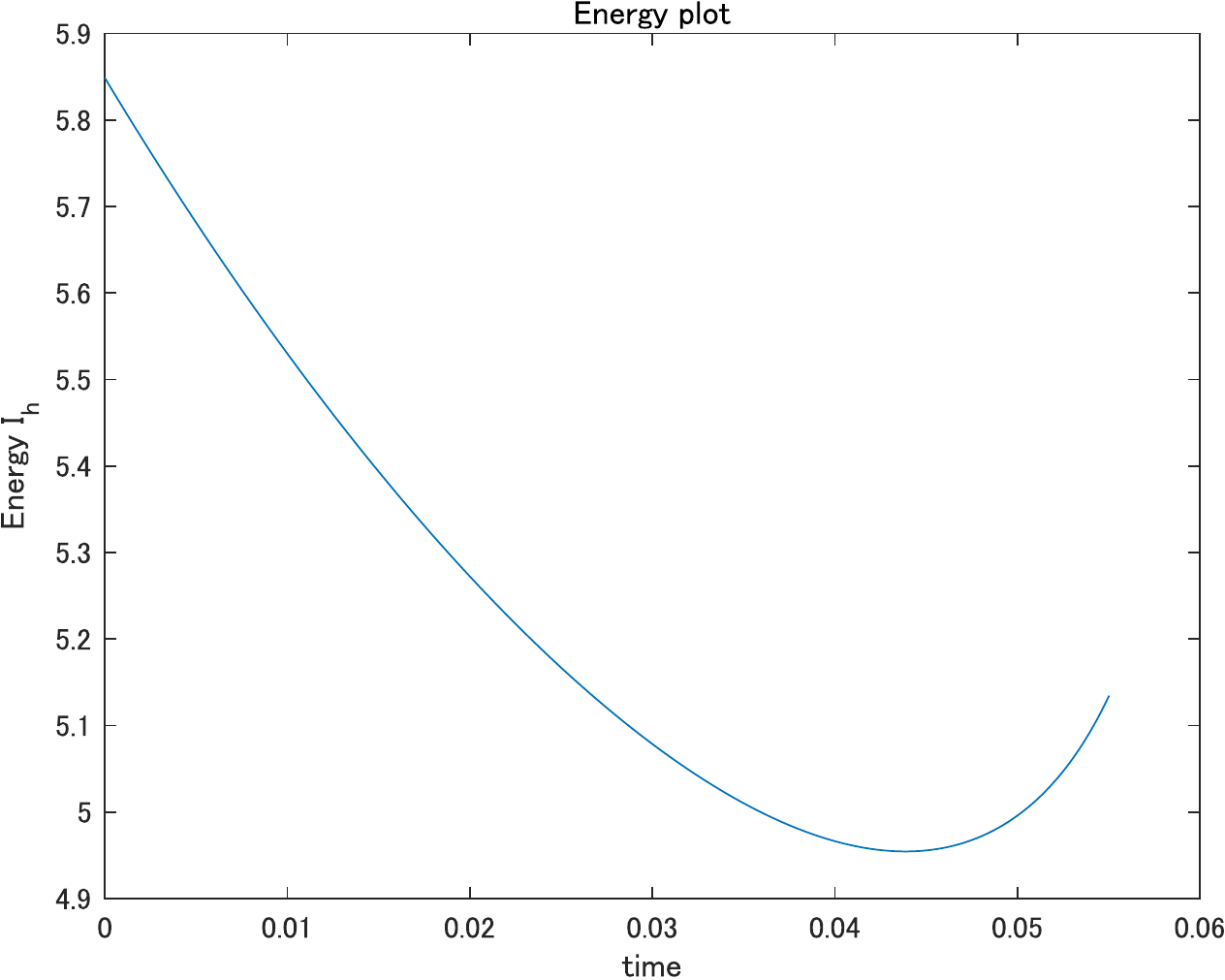}\\
         \textup{(ML--2)} \& $u(0,x)=13\cos\frac{\pi}{2}x$
        \end{center}
      \end{minipage}
     \caption{\tn{Temporal dynamics of the energy} functional $I_{h}(u_{h}^{n})$ \tn{in scheme (ML--2) with small (left) and large (right) data inputs}}
\label{fig:6}
\end{figure}

Finally, we \tn{calculated the} numerical blow-up time \tn{in scheme} (ML--2).  
\tn{Here, we} set $h=1/m~(m=16,32,64)$. \tn{The time} increments \tn{were} defined as 
\begin{align*}
\mbox{(K) } \tau_{n}&=\frac{1}{N+1}\min\left\{1,\frac{1}{\vnorm{u_{h}^{n}}^{\alpha}}\right\}\quad (\mbox{see Theorem \ref{prop:5.6}}),\\
\mbox{(I) } \tau_{n}&=\frac{1}{N+1}\min\left\{1,\frac{1}{I_{h}(u_{h}^{n})^{\alpha}}\right\}\quad (\mbox{see Theorem \ref{prop:5.7}}).
\end{align*} 
For \tn{a} comparison \tn{evaluation}, we \tn{executed} 
the FDM of Chen \cite{che92} and the FDM of Cho and Okamoto \cite{co20}. Specifically, 
set $\tau_{n}=\frac{1}{2} h^{2}\cdot\min\left\{1,\frac{1}{\|u_{h}^{n}\|_{2}^{\alpha}}\right\}$ \tn{in} Chen's FDM and 
$\tau_{n}=\frac{1}{3N}h^{2}\cdot\min\left\{1,\frac{1}{\|u_{h}^{n}\|_{2}^{\alpha}}\right\},~\sigma=\frac{3}{2N}$ \tn{in} Cho and Hamada's FDM. 

We then \tn{introduced} the \emph{truncated numerical blow-up time} $\hat{T}_M (h)$\tn{:}
\begin{equation*}
\hat{T}_M (h)=\min \left\{ t_n \mid \|u_h^n\|_\infty>M=10^{8}\right\}. 
\end{equation*}

\tn{Evaluations were performed with three parameter sets}: 
\begin{description}
 \item[Case 1] 
$N=5$, $\alpha=0.39$, and $u(0,x)=800\cos\frac{\pi}{2}x$; 
\item[Case 2] $N=4$,~$\alpha=0.49$, and $u(0,x)=800(1-x^2)$; 
\item[Case 3] $N=3$,~$\alpha=0.66$, and $u(0,x)=1000(e^{-x^{2}}-e^{-1})$.
\end{description}
\tn{Note that if $1>\frac{N\alpha}{2}$ holds true and $u^{0}\ge0$ is decreasing in $x$, then $I(u(t))$, $\|u(t)\|$ and $\|u(t)\|_{L^{\infty}(I)}$ blow up simultaneously; see
\cite{fm85}.}   
\tn{We chose these settings so that the assumptions in Theorem \ref{prop:5.6} and Theorem \ref{prop:5.7} hold true.}

Fig.~\ref{fig:10} compares the truncated numerical blow-up times $\hat{T}_M (h)$ \tn{as functions of} $h$ \tn{in the four schemes}. 
\tn{The} solution of Chen's FDM blew up \tn{later than the other schemes, whereas} that of (ML--2) with $I_{h}(u_{h}^{n})$ blew up \tn{sooner than the other schemes.}

\begin{figure}[htbp]
      \begin{minipage}{0.32\textwidth}
        \begin{center}
           \includegraphics[width=\textwidth,height=7.5cm]{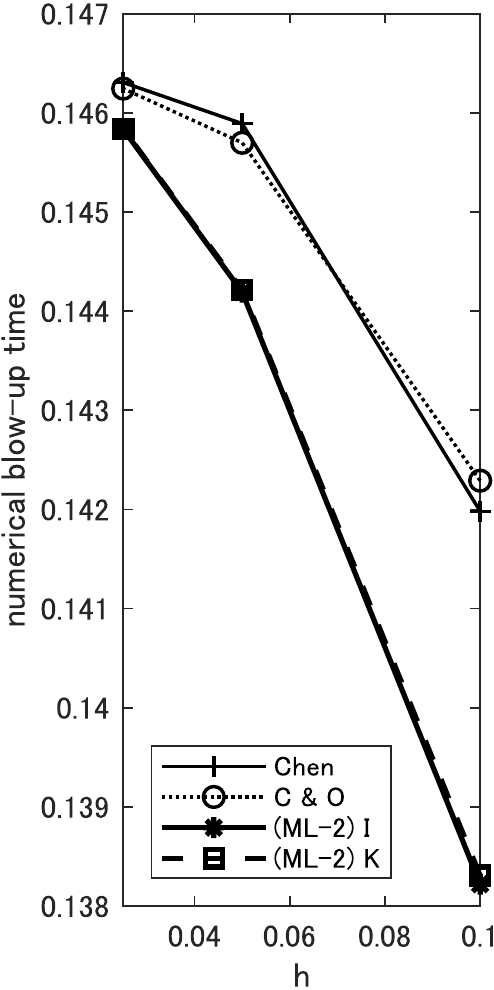} \\
         {Case 1}
        \end{center}
      \end{minipage}
      \begin{minipage}{0.32\textwidth}
        \begin{center}
           \includegraphics[width=\textwidth,height=7.5cm]{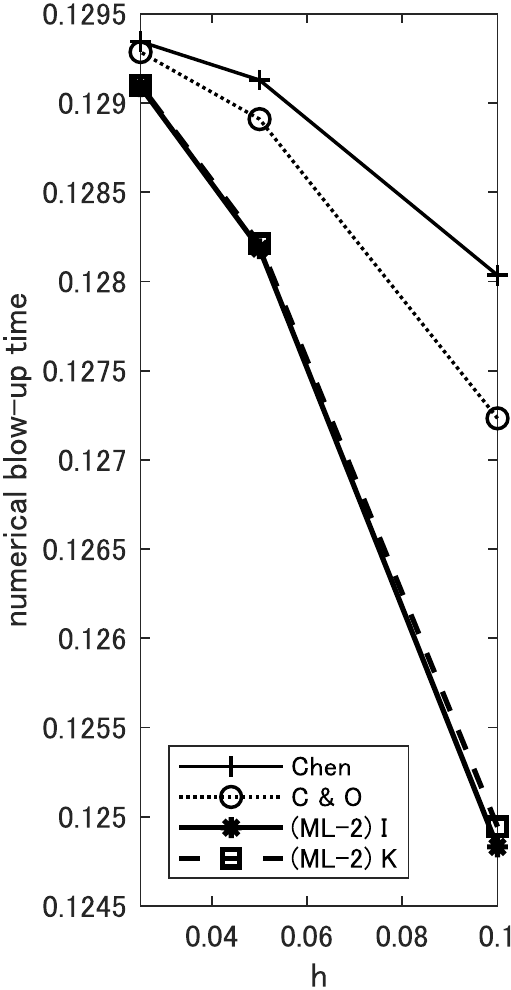}\\
         {Case 2}
\end{center}
      \end{minipage}
\begin{minipage}{0.32\textwidth}
        \begin{center}
           \includegraphics[width=\textwidth,height=7.5cm]{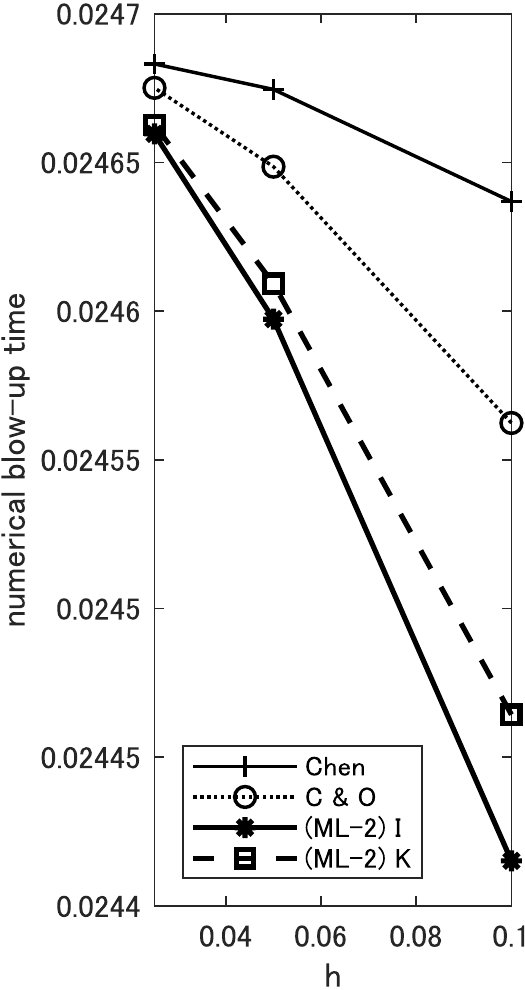}\\
         {Case 3}
        \end{center}
      \end{minipage}
     \caption{\tn{Truncated} numerical blow-up \tn{times} $\hat{T}_M (h)$ \tn{in the four schemes with three parameter settings} }
\label{fig:10}
\end{figure}

\appendix 
\section{Approximate eigenvalue problems}
\label{sec:6}

\tn{This Appendix establishes the proof of} Proposition \ref{conv:eigenvalue-main}. 
Recall that \tn{$\mu$ and $\hat{\mu}_{h}$ are the smallest eigenvalues of \eqref{eq:evp} and \eqref{eigen_h}, respectively}. Functions 
$\psi$ and $\hat{\psi}_h$ are the eigenfunctions \tn{associated} with $\mu$ and $\hat{\mu}_{h}$, respectively.


We introduce \tn{the following} linear operators $T:\dot{H}^{1}\to \dot{H}^{1}$, $T_{h}:S_{h}\to S_{h}$ and $\hat{T}_{h}:S_{h}\to S_{h}$
\begin{align*}
&A(Tv,\chi)=(v,\chi) && (\chi\in \dot{H}^{1},v\in \dot{H}^{1}),\\
&A(T_{h}v_{h},\chi_{h})=(v_{h},\chi_{h}) && (\chi_{h} \in S_{h},v_{h}\in S_{h}),\\
&A(\hat{T}_{h}v_{h},\chi_{h})=\dual{v_{h},\chi_{h}} &&(\chi_{h} \in S_{h},v_{h}\in S_{h}).
\end{align*}


\tn{We also write} $v'=v_x$ for \tn{some} function $v=v(x)$.

\begin{lemma}[\tn{Lemma 2.2 in \cite{se81}, Theorems 13 and 14 in \cite{fri64}}]
\label{lem:regularity}
For any $f_{h}\in S_{h}\subset \dot{H}^{1}$, 
\[
Tf_{h}\in H^{2}(I),~\|(Tf_{h})''\|\le C\|f_{h}\|.
\]
\end{lemma}

For a linear operator $B:X\subset \dot{H}\to \dot{H}$, we define 
\[
 \|B\|_{1,X}=\sup_{v\in X,~v\neq0}\frac{\|(Bv)'\|}{\|v'\|}.
\]
\begin{lemma}[\tn{Lemma~3.3 in \cite{bo90}}]
\label{lem:conv_op}
$\|T-\hat{T}_{h}\|_{1,S_{h}}\to 0$ as $h\to 0$.
\end{lemma}

\begin{rem}
Lemma \ref{lem:conv_op} does not exactly \tn{agree} with \tn{Lemma~3.3 in \cite{bo90}}, \tn{but its proof is identical to} that of \tn{Lemma~3.3 in \cite{bo90}}. 
\end{rem}

\tn{The resolvent operator $R_{z}(\hat{T}_{h})$ for $z\in \mathbb{C}$ is defined as} 
\[
R_{z}(\hat{T}_{h})=(zI-\hat{T}_{h})^{-1}:S_{h}\to S_{h},
\]
where $\hat{T}_{h}:S_{h}\to S_{h}$ and $z\in \rho(\hat{T}_{h})$. 
\tn{Let} $I$ \tn{be} the identity operator and $\rho(B)$ \tn{be} the resolvent set of a linear operator $B$.

\begin{lemma}[{\cite[Lemma 1]{dnr78}}]
\label{lem:resolvent}
For any closed set $F\subset\rho(T)$, there exists $h_{0}>0$ such that for any $h\le h_{0}$, $R_{z}(\hat{T}_{h})$ exists and 
\[
\|R_{z}(\hat{T}_{h})\|_{1,S_{h}}\le C\quad (\forall z\in F),
\]
where $C$ is independent of $h$. 
\end{lemma}

We \tn{now} define spectral projections of $T$ and $\hat{T}_{h}$.
Let $\Gamma\subset \mathbb{C}$ be a circle centered at $\frac{1}{\mu}$ \tn{enclosing} no other points of $\sigma(T)$ which stands for the spectral set of $T$.
Let $E=E(\frac{1}{\mu}):\dot{H}^{1}\to\dot{H}^{1}$ and $\hat{E}_{h}=\hat{E}_{h}(\frac{1}{\mu}):S_{h}\to S_{h}$ be the spectral projection operators
associated with $T$ and $\hat{T}_{h}$ and the parts of the corresponding spectrum enclosed by $\Gamma$, respectively:
\[
E=\frac{1}{2\pi i}\int_{\Gamma}R_{z}(T)~dz,\quad 
\hat{E}_{h}=\frac{1}{2\pi i}\int_{\Gamma}R_{z}(\hat{T}_{h})~dz.
\]

\begin{rem}
By Lemma \ref{lem:resolvent}, \tn{when $h$ is sufficiently small}, $\Gamma\subset\rho(\hat{T}_{h})$ holds and
$\|R_{z}(\hat{T}_{z})\|_{1,S_{h}}$ is bounded for all $z\in\Gamma$.
Thus the integral of $\hat{E}_{h}$ exists.
\end{rem}

\tn{To} examine the convergence property of $\hat{E}_{h}$\tn{, we use the following lemma.}
\begin{lemma}[\tn{Lemma 2 in \cite{dnr78}}]
$\|E-\hat{E}_{h}\|_{1,S_{h}}\to 0$ as $h\to 0$.
\end{lemma}

We use the following symbols. 
\begin{itemize}
\item $\dist(w,A)=\dps\inf_{y\in A}\|w'-y'\|$\quad ($w\in\dot{H}^{1}$, $A\subset \dot{H}^{1}$), 
\item $\delta(\hat{E}_{h}(S_{h}),E(\dot{H}^{1}))=\dps\sup_{v_{h}\in\hat{E}_{h}(S_{h}),~\|v'_{h}\|=1}\dist(v_{h},E(\dot{H}^{1}))$, 
\item $\delta(E(\dot{H}^{1}),\hat{E}_{h}(S_{h}))=\dps\sup_{v\in E(\dot{H}^{1}),~\|v'\|=1}\dist(v,\hat{E}_{h}(S_{h}))$, 
\item $\hat{\delta}(E(\dot{H}^{1}),\hat{E}_{h}(S_{h}))=\max\{{\delta(\hat{E}_{h}(S_{h}),E(\dot{H}^{1})),~\delta(E(\dot{H}^{1}),\hat{E}_{h}(S_{h}))}\}$. 
\end{itemize}

The \tn{next result} follows from the property of \tn{the} spectral projection operator.  

\begin{cor}
\label{cor:gap1}
$\delta(\hat{E}_{h}(S_{h}),E(\dot{H}^{1}))\to 0$ as $h\to 0$.
\end{cor}

\begin{lemma}[\tn{Theorem 2 in \cite{dnr78}}]
\label{lem:approximate_func}
\[
\lim_{h\to0}\inf_{\chi_{h}\in S_{h}}\|(u-\chi_{h})'\|=0\quad (u\in \dot{H}^{1}).
\] 
\end{lemma}

\begin{cor}[\tn{Theorem 3 in \cite{dnr78}}]
$\delta(E(\dot{H}^{1}),\hat{E}_{h}(S_{h}))\to 0$ as $h\to 0$.
\end{cor}

\begin{lemma}[\tn{Corollary 2.6 in \cite{kat95}}]
If $\hat{\delta}(E(\dot{H}^{1}),\hat{E}_{h}(S_{h}))<1$, then
$\dim E(\dot{H}^{1}) = \dim \hat{E}_{h}(S_{h})$. 
\end{lemma}

\tn{For} sufficiently small $h$, we \tn{observe that} $\hat{\delta}(E(\dot{H}^{1}),\hat{E}_{h}(S_{h}))<1$\tn{; that is,} $\dim E(\dot{H}^{1}) = \dim \hat{E}_{h}(S_{h})$. 

Therefore, $\dim E(\dot{H}^{1})=1$, because $E(\dot{H}^{1})$ is the \tn{eigenspace of} the smallest eigenvalue of \eqref{eq:evp}. Then, the unique eigenvalue of $\hat{T}_{h}$ \tn{(denoted by $\frac{1}{\hat{\nu}_{h}}$)} is located inside $\Gamma$.
Then, there exists $\hat{\xi}_{h}(\neq0)\in S_{h}$ such that
\[
A(\hat{\xi}_{h},\chi_{h})=\hat{\nu}_{h}\dual{\hat{\xi}_{h},\chi_{h}},~\chi_{h}\in S_{h}.
\]

\begin{cor}
\label{conv:eigenvalue}
${\hat{\nu}}_{h}\to \mu$ as $h\to 0$. 
\end{cor}

\begin{proof}
\tn{For some arbitrary} $\epsilon>0$, we set $\Gamma_{\epsilon}=B_{\frac{1}{\mu}}(\epsilon)=\{z\in\mathbb{C} \mid  |z-\frac{1}{\mu}|=\epsilon\}$. 
As \tn{stated} above, there exists $h_{\epsilon}>0$ such that \tn{the} eigenvalue $\frac{1}{{\hat{\nu}}_{h}}$ of $\hat{T}_{h}$ is inside $\Gamma_{\epsilon}$ for all $h<h_{\epsilon}$.
Therefore\tn{,} 
\[
\left|\frac{1}{\hat{\nu}_{h}}-\frac{1}{\mu}\right|\le\epsilon.
\]
Because $\mu$ is positive, the proof \tn{is complete}. 
\end{proof}

\begin{rem}
\tn{Similarly}, we \tn{find that a} unique eigenvalue $\frac{1}{\nu_{h}}$ of $T_{h}$ \tn{exists} inside $\Gamma$, and \tn{that} $\nu_h\to \mu$ as $h \to 0$. 
\end{rem}

Let $\mu_{h}>0$ be the smallest eigenvalue of 
\begin{equation}
\label{eigen:app_h}
A(\psi_{h},\chi_{h})=\mu_{h}(\psi_{h},\chi_{h}),~\chi\in S_{h},
\end{equation}
where~$\psi_{h}\in S_{h}$.

\begin{lemma}
For sufficiently small $h>0$, we have $\nu_{h}=\mu_{h}$. 
In particular, $\mu_{h}\to \mu$ as $h\to 0$.
\end{lemma}

\begin{proof}
We know \tn{that} $\dim E(\dot{H}^{1})=1$. By variational characterization, \tn{we obtain}
\[
\mu=\inf_{v\in\dot{H}^{1},v\neq0}\frac{\|v'\|^{2}}{\|v\|^{2}}\le \inf_{v_{h}\in S_{h},v_{h}\neq0}\frac{\|v_{h}'\|^{2}}{\|v_{h}\|^{2}}.
\]
\tn{Here} $\mu_{h}$ \tn{is} the smallest eigenvalue of (\ref{eigen:app_h}), that is,
\[
\mu_{h}=\inf_{v_{h}\in S_{h},v_{h}\neq0}\frac{\|v_{h}'\|^{2}}{\|v_{h}\|^{2}}.
\]
\tn{All} eigenvalues of (\ref{eigen:app_h}) are greater than or equal to~$\mu$.
\tn{As the} eigenvalue of~$T_{h}$~enclosed by~$\Gamma$ \tn{is unique},~we obtain~$\nu_{h}=\mu_{h}$~for sufficiently small $h>0$\tn{,}~and~$\mu_{h}\to\mu$ as $h\to 0$.
\end{proof}

\tn{We now state} the following proof. 

\begin{proof}[Proof of \tn{Proposition} \ref{conv:eigenvalue-main} \textup{(i)}]
By variational characterization, we get
\begin{align*}
\hat{\mu}_{h}&=\inf_{v_{h}\in S_{h},v_{h}\neq0}\frac{\|v'_{h}\|^{2}}{\vnorm{v_{h}}^{2}}=\left(\sup_{v_{h}\in S_{h},v_{h}\neq0}\frac{A(\hat{T}_{h}v_{h},v_{h})}{\|v'_{h}\|^{2}}\right)^{-1},\\
\mu_{h}&=\inf_{v_{h}\in S_{h},v_{h}\neq0}\frac{\|v'_{h}\|^{2}}{\|v_{h}\|^{2}}=\left(\sup_{v_{h}\in S_{h},v_{h}\neq0}\frac{A(T_{h}v_{h},v_{h})}{\|v'_{h}\|^{2}}\right)^{-1}.
\end{align*}
Then, 
\begin{align*}
\sup_{v_{h}\in S_{h},v_{h}\neq0}\frac{A(\hat{T}_{h}v_{h},v_{h})}{\|v'_{h}\|^{2}}&=\sup_{v_{h}\in S_{h},v_{h}\neq0}\left(\frac{A(T_{h}v_{h},v_{h})}{\|v'_{h}\|^{2}}+\frac{A((\hat{T}_{h}-T_{h})v_{h},v_{h})}{\|v'_{h}\|^{2}}\right)\\
&\le \sup_{v_{h}\in S_{h},v_{h}\neq0}\frac{A(T_{h}v_{h},v_{h})}{\|v'_{h}\|^{2}} +\|\hat{T}_{h}-T_{h}\|_{1,S_{h}}.
\end{align*}
Similarly,
\[
\sup_{v_{h}\in S_{h},v_{h}\neq0}\frac{A(T_{h}v_{h},v_{h})}{\|v'_{h}\|^{2}}\le\sup_{v_{h}\in S_{h},v_{h}\neq0}\frac{A(\hat{T}_{h}v_{h},v_{h})}{\|v'_{h}\|^{2}} +\|\hat{T}_{h}-T_{h}\|_{1,S_{h}}.
\]
\tn{Applying} Lemma \ref{lem:conv_op}\tn{, we} obtain 
$\hat{\mu}_{h}=\mu$ as $h\to 0$. 
\end{proof}

\begin{rem}
\tn{As the} eigenvalue of $\hat{T}_{h}$ enclosed by $\Gamma$ \tn{is unique}, we conclude \tn{that} $\hat{\nu}_{h}=\hat{\mu}_{h}$ for sufficiently small $h>0$.
\end{rem}

%

\begin{proof}[Proof of \tn{Proposition} \ref{conv:eigenvalue-main} \textup{(ii)}]
We \tn{write} \eqref{eigen_h} \tn{in} matrix form:
\[
\mathcal{A}\bm{\psi}=\hat{\mu}_{h}\mathcal{M}\bm{\psi},
\]
where 
$\mathcal{M}=\operatorname{diag}(\mu_{i})_{0\le i\le m-1}$ and 
$\mathcal{A}=(a_{i,j})_{0\le i,j\le m-1}$ are \tn{defined} as 
$m_i=(1,\phi_i)$ and $a_{i,j}=A(\phi_j,\phi_i)$\tn{, respectively}. 
Moreover, $\bm{\psi}=(\psi_{i})_{0\le i\le m-1}\in\mathbb{R}^{m},~\psi_{i}=\hat{\psi}_{h}(x_{i})$. 
\tn{We thus} express \eqref{eigen_h} as  
\[
\mathcal{M}^{-1}\mathcal{A}\bm{\psi}=\hat{\mu}_{h}\bm{\psi}.
\]
Because $\mathcal{M}$ is a diagonal matrix, the diagonal components $a_{i,i}/\mu_{i,i}$ of $\mathcal{M}^{-1}\mathcal{A}$ are all positive and \tn{the} non-diagonal components are \tn{all} non-positive. 

\tn{Writing}
\begin{equation}
\label{eq:mat-vec}
\left(-\mathcal{M}^{-1}\mathcal{A}+\max_{0\le i\le m-1}\frac{a_{i,i}}{m_{i,i}}\mathcal{I}\right)\bm{\psi}=\left(-\hat{\mu}_{h}+\max_{0\le i\le m-1}\frac{a_{i,i}}{m_{i,i}}\right)\bm{\psi},
\end{equation}
\tn{where} $\mathcal{I}\in\mathbb{R}^{m\times m}$ is the identity matrix, 
then all components of the matrix \tn{on} the left-hand side are non-negative. 
We \tn{now} consider the following eigenvalue problem:
\begin{equation}
\label{eq:new_eigen}
\left(-\mathcal{M}^{-1}\mathcal{A}+\max_{0\le i\le m-1}\frac{a_{i,i}}{m_{i,i}}\mathcal{I}\right)\vec{x}
=\tilde{\mu}_{h}\vec{x}.
\end{equation}
By the Perron--Frobenius theorem, we can take a positive eigenvector for the largest eigenvalue of \eqref{eq:new_eigen}. $-\hat{\mu}_{h}+\dps\max_{0\le i\le m-1}\frac{a_{i,i}}{m_{i,i}}$ is the largest eigenvalue of \eqref{eq:new_eigen}, because $\hat{\mu}_{h}$ is the smallest eigenvalue of \eqref{eigen_h}.

Consequently, the \tn{sign of the} first eigenfunction of \eqref{eigen_h} \tn{is unchanged}.
\end{proof}

\begin{proof}[Proof of \tn{Proposition} \ref{conv:eigenvalue-main} \textup{(iii)}]
We assume \tn{that}
\[
\hat{\psi}_{h}\ge 0\quad \mbox{and}\quad \int_{I}x^{N-1}\hat{\psi}_{h}(x)~dx=1.
\]
 
\tn{Setting} $\phi=\psi/\|\psi'\|$ and $\hat{\phi}_{h}=\hat{\psi}_{h}/\|\hat{\psi}'_{h}\|$,
\tn{applying} Corollary \ref{cor:gap1} and Proposition \ref{conv:eigenvalue-main}~(i), and setting $v_{h}=\hat{\phi}_{h}$, we obtain
\[
\dist(\hat{\phi}_{h},E(\dot{H}^{1}))\to 0 \quad \mbox{as}\quad h\to 0.
\] 
Because $\dim E(\dot{H}^{1})=1$ and $E(\dot{H}^{1})$ is \tn{a} closed subspace in $\dot{H}^{1}$, we find that
\[
E(\dot{H}^{1})=\{z\phi\in\dot{H}^{1} \mid  z\in\mathbb{C} \}
\]
and
\[
\dist(\hat{\phi}_{h},E(\dot{H}^{1}))=\|\hat{\phi}'_{h}-c_{h}\phi'\|,
\]
where $c_{h}\in\mathbb{C}$.

Therefore, $|c_{h}|=1$ as $h\to 0$. 
Using $\hat{\phi}_{h},~\phi\ge 0$ and $\|\hat{\phi}_{h}-c_{h}\phi\|\le\|\hat{\phi}'_{h}-c_{h}\phi'\|$, we \tn{find that} 
$c_{h}=1$ as $h\to 0$.
That is, as $h\to 0$, 
\begin{align*}
\|\hat{\phi}'_{h}-\phi'\|&\le\|\hat{\phi}'_{h}-c_{h}\phi'\| + \|c_{h}\phi'-\phi'\|\\
       &=\|\hat{\phi}'_{h}-c_{h}\phi'\|+|c_{h}-1|\cdot\|\phi'\|\to 0.
\end{align*}

On the other hand, 
\[
 \|(\psi-\hat{\psi}_{h})'\|\le 
\frac{1}{\int_{I}x^{N-1}\phi(x)~dx}\|\phi'-\hat{\phi}'_{h}\|\\
+\left|\frac{1}{\int_{I}x^{N-1}\phi(x)~dx}-\frac{1}{\int_{I}x^{N-1}\hat{\phi}_{h}(x)~dx}\right|.
\]
This, together with 
\begin{multline*}
 \left|\int_{I}x^{N-1}\phi(x)~dx-\int_{I}x^{N-1}\hat{\phi}_{h}(x)~dx\right| \\
\le \left(\frac{1}{N}\right)^{\frac{1}{2}}\cdot\|\phi-\hat{\phi}_{h}\|
\le\left(\frac{1}{N}\right)^{\frac{1}{2}}\cdot\|\phi'-\hat{\phi}'_{h}\|\to 0 \quad \mbox{as}\quad h\to 0
\end{multline*}
implies that $\|(\psi-\hat{\psi}_{h})'\|\to 0$ as $h\to 0$, \tn{which} completes the proof. 
\end{proof}

\paragraph{Acknowledgments.}
This work was supported by JST CREST Grant Number JPMJCR15D1, Japan, and JSPS KAKENHI Grant Number 15H03635, Japan. In addition, the first author was supported by the Program for Leading Graduate Schools, MEXT, Japan.


\begin{thebibliography}{10}

\bibitem{alm01}
L.~M. Abia, J.~C. L\'{o}pez-Marcos, and J.~Mart\'{\i}nez.
\newblock The {E}uler method in the numerical integration of reaction-diffusion
  problems with blow-up.
\newblock {\em Appl. Numer. Math.}, 38(3):287--313, 2001.

\bibitem{bo90}
U.~Banerjee and J.~E. Osborn.
\newblock Estimation of the effect of numerical integration in finite element
  eigenvalue approximation.
\newblock {\em Numer. Math.}, 56(8):735--762, 1990.

\bibitem{bgr04}
C.~Br\"{a}ndle, P.~Groisman, and J.~D. Rossi.
\newblock Fully discrete adaptive methods for a blow-up problem.
\newblock {\em Math. Models Methods Appl. Sci.}, 14(10):1425--1450, 2004.

\bibitem{bqr05}
C.~Br\"{a}ndle, F.~Quir\'{o}s, and J.~D. Rossi.
\newblock An adaptive numerical method to handle blow-up in a parabolic system.
\newblock {\em Numer. Math.}, 102(1):39--59, 2005.

\bibitem{cgkm16}
A.~Cangiani, E.~H. Georgoulis, I.~Kyza, and S.~Metcalfe.
\newblock Adaptivity and blow-up detection for nonlinear evolution problems.
\newblock {\em SIAM J. Sci. Comput.}, 38(6):A3833--A3856, 2016.

\bibitem{che86}
Y.~G. Chen.
\newblock Asymptotic behaviours of blowing-up solutions for finite difference
  analogue of {$u_t=u_{xx}+u^{1+\alpha}$}.
\newblock {\em J. Fac. Sci. Univ. Tokyo Sect. IA Math.}, 33(3):541--574, 1986.

\bibitem{che92}
Y.~G. Chen.
\newblock Blow-up solutions to a finite difference analogue of {$u_t=\Delta
  u+u^{1+\alpha}$} in {$N$}-dimensional balls.
\newblock {\em Hokkaido Math. J.}, 21(3):447--474, 1992.

\bibitem{cho10}
C.~H. Cho.
\newblock A finite difference scheme for blow-up solutions of nonlinear wave
  equations.
\newblock {\em Numer. Math. Theory Methods Appl.}, 3(4):475--498, 2010.

\bibitem{cho13}
C.~H. Cho.
\newblock On the finite difference approximation for blow-up solutions of the
  porous medium equation with a source.
\newblock {\em Appl. Numer. Math.}, 65:1--26, 2013.

\bibitem{cho16}
C.~H. Cho.
\newblock Numerical detection of blow-up: a new sufficient condition for
  blow-up.
\newblock {\em Jpn. J. Ind. Appl. Math.}, 33(1):81--98, 2016.

\bibitem{cho18}
C.~H. Cho.
\newblock On the computation for blow-up solutions of the nonlinear wave
  equation.
\newblock {\em Numer. Math.}, 138(3):537--556, 2018.

\bibitem{cho07}
C.~H. Cho, S.~Hamada, and H.~Okamoto.
\newblock On the finite difference approximation for a parabolic blow-up
  problem.
\newblock {\em Japan J. Indust. Appl. Math.}, 24(2):131--160, 2007.

\bibitem{co20}
C.~H. Cho and H.~Okamoto.
\newblock Finite difference schemes for an axisymmetric nonlinear heat equation
  with blow-up.
\newblock {\em Electron. Trans. Numer. Anal.}, 52:391--415, 2020.

\bibitem{dl00}
K.~Deng and H.~A. Levine.
\newblock The role of critical exponents in blow-up theorems: the sequel.
\newblock {\em J. Math. Anal. Appl.}, 243(1):85--126, 2000.

\bibitem{dnr78}
J.~Descloux, N.~Nassif, and J.~Rappaz.
\newblock On spectral approximation. {I}. {T}he problem of convergence.
\newblock {\em RAIRO Anal. Num\'{e}r.}, 12(2):97--112, iii, 1978.

\bibitem{et84}
K.~Eriksson and V.~Thom\'{e}e.
\newblock Galerkin methods for singular boundary value problems in one space
  dimension.
\newblock {\em Math. Comp.}, 42(166):345--367, 1984.

\bibitem{fgr03}
R.~Ferreira, P.~Groisman, and J.~D. Rossi.
\newblock Adaptive numerical schemes for a parabolic problem with blow-up.
\newblock {\em IMA J. Numer. Anal.}, 23(3):439--463, 2003.

\bibitem{fri64}
A.~Friedman.
\newblock {\em Partial differential equations of parabolic type}.
\newblock Prentice-Hall, Inc., Englewood Cliffs, N.J., 1964.

\bibitem{fm85}
A.~Friedman, B.~McLeod.
\newblock Blow-up of positive solutions of semilinear heat equations.
\newblock {\em Indiana Univ. Math. J.}, 34:425--447, 1985. 

\bibitem{gro06}
P.~Groisman.
\newblock Totally discrete explicit and semi-implicit {E}uler methods for a
  blow-up problem in several space dimensions.
\newblock {\em Computing}, 76(3-4):325--352, 2006.

\bibitem{gr01}
P.~Groisman and J.~D. Rossi.
\newblock Asymptotic behaviour for a numerical approximation of a parabolic
  problem with blowing up solutions.
\newblock {\em J. Comput. Appl. Math.}, 135(1):135--155, 2001.

\bibitem{hmr08}
W.~Huang, J.~Ma, and R.~D. Russell.
\newblock A study of moving mesh {PDE} methods for numerical simulation of
  blowup in reaction diffusion equations.
\newblock {\em J. Comput. Phys.}, 227(13):6532--6552, 2008.

\bibitem{iho02}
T.~Ide, C.~Hirota, and M.~Okada.
\newblock Generalized energy integral for {${\partial u\over\partial t}={\delta
  G\over\delta u}$}, its finite difference schemes by means of the discrete
  variational method and an application to {F}ujita problem.
\newblock {\em Adv. Math. Sci. Appl.}, 12(2):755--778, 2002.

\bibitem{ish10}
M.~Ishiwata.
\newblock On the asymptotic behavior of unbounded radial solutions for
  semilinear parabolic problems involving critical {S}obolev exponent.
\newblock {\em J. Differential Equations}, 249(6):1466--1482, 2010.

\bibitem{it00}
T.~Ishiwata and M.~Tsutsumi.
\newblock Semidiscretization in space of nonlinear degenerate parabolic
  equations with blow-up of the solutions.
\newblock {\em J. Comput. Math.}, 18(6):571--586, 2000.

\bibitem{jes78}
D.~Jespersen.
\newblock Ritz-{G}alerkin methods for singular boundary value problems.
\newblock {\em SIAM J. Numer. Anal.}, 15(4):813--834, 1978.

\bibitem{kat95}
T.~Kato.
\newblock {\em Perturbation theory for linear operators}.
\newblock Classics in Mathematics. Springer-Verlag, Berlin, 1995.
\newblock Reprint of the 1980 edition.

\bibitem{lev90}
H.~A. Levine.
\newblock The role of critical exponents in blowup theorems.
\newblock {\em SIAM Rev.}, 32(2):262--288, 1990.

\bibitem{nak76}
T.~Nakagawa.
\newblock Blowing up of a finite difference solution to
  {$u_{t}=u_{xx}+u^{2}$}.
\newblock {\em Appl. Math. Optim.}, 2(4):337--350, 1975/76.

\bibitem{nu77}
T.~Nakagawa and T.~Ushijima.
\newblock Finite element analysis of the semi-linear heat equation of blow-up
  type.
\newblock In {\em Topics in Numerical Analysis III}, pages 275--291. Academic
  Press, New York, 1977.

\bibitem{ns20}
T.~Nakanishi and N.~Saito.
\newblock Finite element method for radially symmetric solution of a
  multidimensional semilinear heat equation.
\newblock {\em Jpn. J. Ind. Appl. Math.}, 37(1):165--191, 2020.

\bibitem{nb08}
F.~K. N'Gohisse and T.~K. Boni.
\newblock Numerical blow-up solutions for some semilinear heat equations.
\newblock {\em Electron. Trans. Numer. Anal.}, 30:247--257, 2008.

\bibitem{qs19}
P.~Quittner and P.~Souplet.
\newblock {\em Superlinear parabolic problems}.
\newblock Birkh\"{a}user Advanced Texts: Basler Lehrb\"{u}cher. [Birkh\"{a}user
  Advanced Texts: Basel Textbooks]. Birkh\"{a}user/Springer, Cham, 2019.

\bibitem{ss16}
N.~Saito and T.~Sasaki.
\newblock Blow-up of finite-difference solutions to nonlinear wave equations.
\newblock {\em J. Math. Sci. Univ. Tokyo}, 23(1):349--380, 2016.

\bibitem{ss16b}
N.~Saito and T.~Sasaki.
\newblock Finite difference approximation for nonlinear {S}chr\"{o}dinger
  equations with application to blow-up computation.
\newblock {\em Jpn. J. Ind. Appl. Math.}, 33(2):427--470, 2016.

\bibitem{se81}
R.~Schreiber and S.~C. Eisenstat.
\newblock Finite element methods for spherically symmetric elliptic equations.
\newblock {\em SIAM J. Numer. Anal.}, 18(3):546--558, 1981.

\bibitem{tho06}
V.~Thom{\'e}e.
\newblock {\em Galerkin finite element methods for parabolic problems}.
\newblock Springer Verlag, Berlin, second edition, 2006.

\end{thebibliography}

\end{document}